%% file: main.tex
\documentclass{article}
\input{macros}
\input{packages}

\usepackage{soul}

\title{A reduced-order model for parametrized Optimal Transport problems}
\author{Elise Bonnet-Weill\thanks{ CERMICS, CNRS, ENPC, Institut Polytechnique de Paris, ParMA, Inria Saclay, Marne-la-Vallée, France.
email : elise.bonnet-weill@enpc.fr},
Virginie Ehrlacher\thanks{ CERMICS, CNRS, ENPC, Institut Polytechnique de Paris, Marne-la-Vallée, France}, Luca Nenna\thanks{Universit\'e Paris-Saclay, CNRS, Laboratoire de math\'ematiques d'Orsay, ParMA, Inria Saclay, 91405, Orsay, France. 
email: luca.nenna@universite-paris-saclay.fr} \thanks{Institut Universitaire de France (IUF).}}
\date{\today}

\begin{document}

\maketitle

\begin{abstract}
    In this work, we aim at efficiently solving a parametrized family of optimal transport problems by using model order reduction methods. We propose a reduced-order model by adding to the primal (respectively dual) version of the high-fidelity model the additional constraint to live in a non negative sub cone (resp. in subspaces) of small dimension. The reduced-order model then reads as a linear program with a small number of degrees of freedom and constraints.  We identify explicit conditions under which this reduced-order model has at least one solution. We propose two a posteriori error estimations that bounds the error between the optimal values of the high-fidelity problem and the reduced-order model. As one of these estimations requires the computation of non linear terms (with respect to the reduction of dimension), we use an Empirical Interpolation Method (EIM) (see e.g. \cite{maday2007general} or \cite{barrault2004empirical}) to numerically efficiently compute  this estimation. We apply the whole methodology on a simple 1D example and on a problem of color transfer between images, and compare its performances to Sinkhorn algorithm. 
\end{abstract}

\section*{Notations}
We present here notations that are used throughout the article without recalling their meaning.
\begin{itemize}
    \item Given two real matrices $\mathbf{B}$ and $\mathbf{E}$ of dimension $\cN_x \times \cN_y$, where $\cN_x \in \N^*$ and $\cN_y \in \N^*$ we denote by $\left\langle \textbf{B}, \textbf{E} \right\rangle$ the Frobenius scalar product between $\textbf{E}$ and $\textbf{B}$.
    \item For $\cN \in \N^*$, we define the $\ell_1$ norm on $\RR^{\cN}$ saying that for every vector $\ww \in \RR^\cN$,
        \begin{equation*}
            \| \ww \|_1 := \sum_{i = 1}^{\cN} |\ww_i|.
        \end{equation*}
    \item For $\cN \in \N^*$, we define the $\ell_\infty$ norm on $\RR^{\cN}$ saying that for every vector $\ww \in \RR^\cN$,
        \begin{equation*}
            \| \ww \|_\infty := \max_{1 \leq i \leq \cN} |\ww_i|.
        \end{equation*}
    \item For a real number $x \in \RR $, $\lfloor x \rfloor$ designates the integer part of $x$.
\end{itemize}

\section{Introduction}
The numerical simulation of complex physical, economic, or biological systems often requires the repeated solution of parametric problems under varying conditions — a task that rapidly becomes computationally prohibitive when state-of-the-art discretization methods are employed. In this context, Model Order Reduction (MOR) has emerged as a fundamental strategy to construct low-dimensional surrogate models that accurately reproduce the behavior of high-fidelity numerical simulations at a fraction of their computational cost. Among the various MOR strategies, Reduced Basis (RB) methods have established themselves as a mathematically rigorous and computationally efficient framework, particularly well-suited to parametric partial differential equations (PDEs). The aim of the present work is to propose a new reduced basis strategy to build reduced-order models to accelerate the resolution of parametrized Optimal Transport problems. We first give some introduction to reduced basis methods and numerical approaches for the resolution of optimal transport problems and then outline the main contributions of our paper. 

\paragraph{Reduced Basis Methods.}
The reduced basis methodology was firstly introduced in the pioneering works of~\cite{almroth1978automatic} and~\cite{noor1980reduced} in the context of structural mechanics, and  subsequently  made mathematically rigorous through the contributions of~\cite{prud2002reliable}. 
The central idea is to approximate the solution manifold — the set of all solutions to a parametric problem as the parameter varies over its admissible domain — by a low-dimensional affine subspace, the so-called reduced basis space, constructed from a carefully selected set of high-fidelity snapshot solutions. 
The RB approximation is then sought in this surrogate space via a Galerkin or Petrov–Galerkin projection, yielding a system of dramatically reduced size that can be solved orders of magnitude faster than the original truth discretization. A theoretical foundation for the convergence of the greedy construction algorithm was later laid by~\cite{buffa2012priori}, and further refined by~\cite{binev2011convergence}. A cornerstone of the RB methodology is the availability of rigorous and efficiently computable a posteriori error estimators, which certify the accuracy of the reduced solution with respect to the high-fidelity reference and drive the greedy algorithm used for basis construction. Such estimators, typically based on dual norms of residuals, were developed for parabolic problems by~\cite{grepl2005posteriori} and for finite-volume evolution schemes by~\cite{haasdonk2008reduced}. 
These estimators enable an offline-online computational decomposition: the expensive pre-computations (basis generation, affine decomposition of operators) are performed once during an offline phase, while the online phase — consisting in solving the reduced system and evaluating the error bound for a new parameter instance — is performed at marginal cost, independent of the dimension of the high-fidelity space. The method has since been systematically exposed in textbook form by~\cite{hesthaven2016certified} and by~\cite{quarteroni2015reduced}, and its general theory for affinely parametrized elliptic coercive PDEs is thoroughly surveyed in~\cite{rozza2008reduced}. Extensions to non-affinely parametrized problems, such as those arising in geometry optimization, have been addressed through the Empirical Interpolation Method (EIM) introduced by~\cite{barrault2004empirical}, which approximates non-affine coefficient functions by a low-rank collateral expansion and permits the offline-online decomposition to be maintained. Non-affine and nonlinear parametric dependence was treated within the full RB framework by~\cite{grepl2007efficient}. Beyond the classical greedy construction, the Proper Orthogonal Decomposition (POD) provides an alternative, singular-value-decomposition-based approach to extracting dominant modes from a set of snapshots~\cite{volkwein2013proper,haasdonk2013convergence}.

\paragraph{Numerical Methods for Optimal Transport.} Optimal Transport (OT) is a classical mathematical theory, dating back to Monge~\cite{monge1781memoire} and Kantorovich~\cite{kantorovich2006translocation}, that formalizes the problem of finding the most efficient way to redistribute one probability measure into another with respect to a given cost function. The monographs by Villani~\cite{villani2021topics,villani2009optimal} and Santambrogio~\cite{santambrogio2015optimal}, and the computational review by Peyré and Cuturi~\cite{COTFNT}, provide comprehensive accounts of the theory and its modern applications, which span machine learning, image processing, computational fluid mechanics, economics, and biology.

At the heart of the theory lies the Kantorovich dual problem, whose solutions — the Kantorovich potentials — are functions whose gradients (in the quadratic cost case) recover the optimal transport map through Brenier's theorem. The regularity theory of these potentials is well-understood under smoothness assumptions on the source and target densities, tracing back to Caffarelli's work on the Monge–Ampère equation and generalized to a large class of cost functions by~\cite{ma2005regularity}. Quantitative stability of optimal transport maps and their associated potentials with respect to perturbations of the input measures has been studied by~\cite{gigli2011holder} and, in a quantitative and linearized form amenable to data science applications, by Mérigot, Delalande and Chazal~\cite{merigot2020quantitative}. This stability theory is fundamental: it provides the mathematical basis for expecting the map sending a parametric family of measures to the corresponding Kantorovich potentials — to exhibit regularity in the parameter, a prerequisite for any successful model reduction strategy.

From a numerical point of view, the discrete formulation of OT leads to a linear programming (LP) problem — the classical transportation problem — whose computational complexity fast increases with the dimension of the problem. This has motivated the development of more scalable approaches. The entropic regularization of OT, introduced in its modern form in ~\cite{cuturi2013sinkhorn,galichon2022cupid,benamou2015iterative}, replaces the LP by a strictly convex optimization problem solvable via the Sinkhorn–Knopp matrix-scaling algorithm at a speed  faster than LP-based solvers. The resulting Sinkhorn distances have become ubiquitous in machine learning, despite the bias introduced by regularization. De-biased versions, including the Sinkhorn divergences of~\cite{feydy2019interpolating}, restore (some) properties of a genuine distance while retaining computational efficiency.

Alternative numerical approaches include semi-discrete OT methods, which handle the case where one measure is continuous and the other discrete through Laguerre cell decompositions and Newton-type or quasi-Newton solvers, as developed by~\cite{merigot2011multiscale} and~\cite{levy2015numerical}; methods based on the dynamic Benamou–Brenier fluid mechanics formulation~\cite{benamou2000computational}, which reformulates OT as a convex optimization over flows and has led to efficient proximal splitting algorithms; and approaches grounded in the theory of semi-smooth Newton methods for the Monge–Ampère equation. A unified treatment of discretization methods and their algorithms, covering the assignment problem, entropic regularization, and semi-discrete OT, is provided in the handbook chapter by~\cite{merigot2021optimal}. Sliced Wasserstein distances offer yet another computationally attractive alternative based on one-dimensional projections.

\paragraph{Parametric Optimal Transport and Model Reduction.} 
Despite this rich landscape of numerical methods, the problem of computing Kantorovich potentials repeatedly for varying parameter values — arising for instance in uncertainty quantification, shape optimization, multi-marginal OT, or Wasserstein barycenter computation — remains computationally intensive. Each parameter query requires solving an independent optimization problem of large scale, rendering parameter sweeps, sensitivity analyses, or real-time applications out of reach with standard solvers. The object of primary interest in this paper is therefore the parametric map which sends a parametric family of measures to the corresponding family of Kantorovich potentials. This is a fundamentally different perspective from the recent literature on MOR for transport-dominated PDEs~\cite{cagniart2018model,blickhan2024registration}, which uses OT as a tool to improve reduced basis constructions for PDE solutions that are merely dominated by advection; here, the optimal transport problem itself is parametric, and the Kantorovich potential is the quantity of  interest.

This shift in viewpoint introduces specific mathematical challenges. First, Kantorovich potentials are only defined up to an additive constant, so the solution map is multi-valued and must be normalized appropriately for a well-posed reduction problem. Second, the derivation of rigorous a posteriori error estimators for the Kantorovich potential — whose accuracy is typically measured in terms of the duality gap or the Wasserstein distance — requires adapting the residual-based methodology of classical RB methods to the OT setting. While the stability theory of Optimal Transport provides qualitative guidance, quantitative and computable error bounds for reduced approximations of Kantorovich potentials have not yet been established.

Recent efforts have begun to address the broader challenge of combining model reduction with optimal transport, forming a rich and rapidly growing literature. A pioneering contribution in this direction is~\cite{iollo2014advection}, who proposed to use solutions of Monge–Kantorovich mass transfer problems to extract advection modes — low-order representations of transport phenomena that are intrinsically elusive for standard POD-based methods. This seminal work demonstrated that OT provides a geometrically meaningful way to encode advection-dominated dynamics into reduced descriptions, and motivated a subsequent line of research on registration-based model reduction. This registration viewpoint is formalized in~\cite{taddei2020registration} into a general equation-independent framework for parameterized MOR, constructing parameter-dependent bijective mappings that align coherent solution features — shocks, shear layers — into a fixed reference configuration before applying linear compression. We refer the reader to extensions of this registration-based approach using OT in~\cite{iollo2022mapping,cucchiara2024model}. 
Complementing these registration-based approaches, the theoretical groundwork for nonlinear model reduction in Wasserstein spaces was laid in~\cite{ehrlacher2020nonlinear}, using Wasserstein barycenters and tangent-space approximations for conservative PDEs. Several extensions were considered, we refer the reader to~\cite{battisti2023wasserstein,blickhan2024registration, do2025sparse, dalery2026nonlinear, peherstorfer2020model}.  We highlight in particular the work~\cite{khamlich2025optimal} where the authors proposed a ROM framework integrating OT theory and neural-network methods via kernel POD with the Wasserstein distance.

However, all of these contributions — whether they use OT to construct registration mappings for PDE solutions, or address model reduction in Wasserstein spaces for PDE solutions — treat optimal transport as a computational tool for handling transport-dominated PDE fields. None of them addresses the problem of directly and efficiently computing parametric Kantorovich potentials themselves, which is the focus of the present work. The only work we are aware of in this direction is the contribution~\cite{lacombe2023learning} where parametrized optimal transport problems were solved by means of neural networks and machine learning techniques.

\paragraph{Contributions of this work.}

In this paper, we develop a novel reduced basis method for parametric optimal transport problems. More precisely, given a parametric family of optimal transport problems where the measures vary with a parameter, our framework constructs a low-dimensional approximation space for the associated family of Kantorovich potentials through an offline procedure (typically POD).  To the best of our knowledge, the reduced basis approximation of parametric Kantorovich potentials, with certified a posteriori error estimation, has not yet been addressed in the literature.  Our reduced-order model actually reads as a primal-dual approach with a reduced number of degrees of freedom. As a by-product, we also obtain reduced approximations of the parametric optimal transport plans.  The key contributions are: 
\begin{itemize}
    \item[(i)] a well-posed formulation of the parametric Kantorovich dual suited to RB approximation; 
    \item [(ii)] the derivation of computable and rigorous a posteriori error estimators enabling reliable online evaluation;
    \item[(iii)] an efficient offline-online decomposition that makes the online cost independent of the high-fidelity discretization dimension. 
\end{itemize}
   We demonstrate the efficiency and accuracy of the proposed approach on a range of parametric OT benchmarks, achieving speedups of several orders of magnitude, in particular for color transfer applications.

\paragraph{Outline.}
The remainder of the paper is organized as follows. In Section~\ref{section: High-fidelity model}, we present the high-fidelity parametric optimal transport problem to be reduced. In Section~\ref{section: RP}, we present the reduced order model. Section~\ref{section : A posteriori error approximator} is devoted to the derivation of a posteriori error estimators and their mathematical analysis. Lastly, in Section~\ref{sec:num}, we illustrate the efficiency of our approach on various numerical test cases, in particular for color transfer applications.


\section{High-fidelity Optimal Transport Problem}\label{section: High-fidelity model}

\subsection{Presentation of the prototypical problem}

In this section, we introduce the parametrized family of discrete optimal transport problems under consideration. We then highlight the continuous dependence of the optimal values with respect to the parameter.

Let $X= \left\lbrace x_1, \dots , x_{\cN_x}\right\rbrace$ and $Y = \left\lbrace y_1, \dots, y_{\cN_y}\right\rbrace$ be two finite sets of size $\cN_x \in \N ^*$ and $\cN_y \in \N^*$, respectively,  $c:X\times Y\rightarrow \RR_+$ be a continuous  cost function and $\displaystyle{\CC = \left\lbrace \CC_{ij} \right\rbrace_{\footnotesize \begin{matrix} 1 \leq i \leq \cN_x \\ 1\leq j \leq \cN_y \end{matrix}} \in \RR^{\cN_x \times \cN_y}}$ the associated cost matrix: $\CC_{i, j}:= c(x_i, y_j) $ for $1\leq i \leq \cN_x$ and for $1\leq j\leq \cN_y$.  

Let $K_x, K_y \in \N^*$, and let $\displaystyle{}\{\mmu_{k}\}_{1 \leq k \leq K_x}$ and  $\displaystyle{}\{\nnu_{l}\}_{1 \leq l \leq K_y}$ be two sets of probability measures on $X$ and $Y$ defined such that for $1\leq k \leq K_x$, $\displaystyle{} \mmu_{k} = \sum_{i = 1}^{\cN_x} \mmu_{k}^i \delta_{x_i}$ and for $1\leq l \leq K_y$, $\displaystyle{}\nnu_{l} = \sum_{j = 1}^{\cN_y} \nnu_{l}^j \delta_{y_j}$.
We define the set of parameters $\cA$ by:
\begin{equation}
    \cA := \left\lbrace \aalpha = (\aalpha^{x}, \aalpha^{y}) \in \RR_+^{K_x+ K_y}, \sum_{ k = 1}^{K_x}\aalpha_{k}^x = 1, \sum_{l = 1}^{K_y}\aalpha_{l}^y = 1 \right\rbrace.
\end{equation}
For a given parameter $\aalpha = (\aalpha^x, \aalpha^y) \in \cA$, we define the parametrised measures $\mmu:\cA\to\mathcal P(X)$ (respectively $\nnu:\cA\to\mathcal P(Y)$) as the convex combinations of $\{\mmu_{k}\}_{1 \leq k \leq K_x}$ (resp.  $\{\nnu_{l}\}_{1 \leq l \leq K_y}$) with respect to the weights $\aalpha^x$ (resp. $\aalpha^y$), that is
\begin{equation}
    \begin{array}{cc}
        \displaystyle{} \mmu(\aalpha) = \sum_{k = 1}^{K_x} \aalpha^x_{k} \mmu_{k},  \\
         \displaystyle \nnu(\aalpha) = \sum_{l= 1}^{K_y} \aalpha^y_{l} \nnu_{l}. 
    \end{array}
\end{equation}
The parametrized high-fidelity (HF) optimal transport problem writes as follows:
\begin{equation}\label{eq: primal matrix}
\displaystyle  I(\aalpha):= \mathop{\min}_{\begin{array}{cc}
&\PPi \in \RR_+^{\cN_x \times \cN_y} \\
&\displaystyle \sum_{j = 1}^{\cN_y} \PPi_{i,j} = \mmu_i(\aalpha) \text{ for } 1\leq i \leq \cN_x \\
&\displaystyle \sum_{i = 1}^{\cN_x} \PPi_{i,j} = \nnu_j(\aalpha) \text{ for } 1\leq j \leq \cN_y \\
\end{array}
}
\left\langle \PPi, \CC \right\rangle .
\end{equation}
Here, it will  be more convenient to rewrite this problem in the standard form  of  a
linear program, as in \cite{COTFNT}: we cast $\PPi \in \RR^{\cN_x \times \cN_y}$ (respectively $\CC\in \RR^{\cN_x \times \cN_y}$) as a vector $\ppi\in {\RR^{\cN_x \cN_y}}$ (resp. $\cc\in \RR^{\cN_x\cN_y}$) by stacking its columns, and define the matrix $\AA$, encoding the constraints, by
\begin{equation}\label{eq: A}
    \AA = \begin{pmatrix}
            \mathbb{1}^T_{\cN_x} \otimes \textbf{Id}_{\cN_y} \\ 
            \mathbf{Id}_{\cN_x} \otimes \11^T_{\cN_y} 
          \end{pmatrix}, 
\end{equation}
where $\mathbf{Id}_\bullet$ denotes the identity matrix of $\RR^\bullet$, $\mathbb{1}_\bullet$ the one vector of $\RR^\bullet$, and $\otimes$ the Kronecker product. The problem \eqref{eq: primal matrix} re-writes, then,  in its standard form:
\begin{equation}\label{eq: primal}\tag{$P_{\aalpha}$}
\boxed{I(\aalpha):= \mathop{\min}_{\begin{array}{c}
\ppi\in \RR_+^{\cN_x \cN_y}\\
\AA \ppi= \begin{pmatrix} \mmu(\aalpha) \\ \nnu(\aalpha)\end{pmatrix}
\end{array}
}
\left\langle \ppi, \cc \right\rangle .}
\end{equation}
The dual associated problem is : 
\begin{equation}\label{eq: dual}\tag{$D_{\aalpha}$}
\boxed{J(\aalpha) = \mathop{\max}_{ \begin{array}{c}
(\pphi,\ppsi) \in \RR^{\cN_x + \cN_y} \\
\AA^T \begin{pmatrix} \pphi \\ \ppsi \end{pmatrix} \leq \cc \\
\end{array}
} 
\left\langle \pphi, \mmu(\aalpha) \right\rangle + \left\langle \ppsi, \nnu(\aalpha) \right\rangle .}
\end{equation}

The problem \ref{eq: primal} always has solutions, as it is a minimization problem in finite dimensions of a continuous function on a non-empty, closed and bounded set. By strong duality, we deduce that the problem \ref{eq: dual} also has solutions and that for any parameter $\aalpha \in \cA$, the optimal values $I(\aalpha)$ and $J(\aalpha)$ are finite, and $I(\aalpha) = J(\aalpha)$.

\subsection{Regularity of the cost with respect to the parameter \texorpdfstring{$\aalpha$}{}}
We now turn our attention on  the dependence of the optimal values with respect to the parameter $\aalpha$, showing, in particular, that, $I(\aalpha)$ (or $J(\aalpha)$) is  continuous in $\aalpha$. Notice that we will exploit this property in Section \ref{subsection: erorr estimation continuity} to derive an a posteriori error estimation. 
To get the desired result, we first derive a bound on $|I(\aalpha)- I(\aalpha ')|$  which depends on the solutions of the dual problem \ref{eq: dual}  at $\aalpha \in \cA$ and $\aalpha ' \in \cA$.
\begin{prop} \label{prop: alpha dependency - bound with potentials}
    Let $\aalpha \in \cA$ and $\aalpha ' \in \cA$ be two parameters. Let $(\pphi_*(\aalpha), \ppsi_*(\aalpha))$ and $(\pphi_*(\aalpha '), \ppsi_*(\aalpha '))$ be solutions of \eqref{eq: dual} and $(D_{\aalpha'})$ at $\aalpha$ and $\aalpha '$ respectivly. Then:
    \begin{equation*}
        |J(\aalpha) - J(\aalpha ')| \leq C_\varphi \| \mmu(\aalpha) - \mmu(\aalpha ') \|_1 
        + C_\psi  \| \nnu(\aalpha) - \nnu(\aalpha ') \|_1,
    \end{equation*}
    where $C_\varphi=\max\big( \| \pphi_*(\aalpha) \|_\infty , \| \pphi_*(\aalpha ') \|_\infty  \big)$ and $C_\psi=\max\big( \| \ppsi_*(\aalpha) \|_\infty , \| \ppsi_*(\aalpha' ) \|_\infty  \big)$.
\end{prop}
\begin{proof}
    The pair $(\pphi_*(\aalpha), \ppsi_*(\aalpha))$ is optimal for \eqref{eq: dual}, and also admissible for $(D_{\aalpha'})$. Thus
    \begin{equation}\label{eq: proof  3.11}
        J(\aalpha) - J(\aalpha') \leq 
        \left\langle \pphi_*(\aalpha), \mmu(\aalpha) - \mmu(\aalpha ') \right\rangle +
        \left\langle \ppsi_*(\aalpha), \nnu(\aalpha) - \nnu(\aalpha ') \right\rangle.
    \end{equation}
    By analogy, since $(\pphi_*(\aalpha '), \ppsi_*(\aalpha '))$ is optimal for $(D_{\aalpha'})$, and admissible for $(D_{\aalpha})$, we obtain :
    \begin{equation}\label{eq: proof 3.1 2}
        J(\aalpha) - J(\aalpha') \geq 
        \left\langle \pphi_*(\aalpha ' ), \mmu(\aalpha) - \mmu(\aalpha ') \right\rangle +
        \left\langle \ppsi_*(\aalpha '), \nnu(\aalpha) - \nnu(\aalpha ') \right\rangle.
    \end{equation}
    Now, using norm inequalities, we bound from above the absolute value of the right hand side terms of  \eqref{eq: proof  3.11} and \eqref{eq: proof 3.1 2}.

    \begin{align}\label{eq: proof 3.1 3}
        | \left\langle \pphi_*(\aalpha), \mmu(\aalpha) - \mmu(\aalpha ') \right\rangle +
        &\left\langle \ppsi_*(\aalpha), \nnu(\aalpha) - \nnu(\aalpha ') \right\rangle |\leq\\
        &\| \pphi_*(\aalpha) \|_\infty \| \mmu(\aalpha) - \mmu(\aalpha ') \|_1 +  \| \ppsi_*(\aalpha) \|_\infty \| \nnu(\aalpha) - \nnu(\aalpha ') \|_1,
    \end{align}

    and
       \begin{align}\label{eq: proof 3.1 4}
        | \left\langle \pphi_*(\aalpha '), \mmu(\aalpha) - \mmu(\aalpha ') \right\rangle +
        &\left\langle \ppsi_*(\aalpha '), \nnu(\aalpha) - \nnu(\aalpha ') \right\rangle |\leq\\
         &\| \pphi_*(\aalpha ') \|_\infty \| \mmu(\aalpha) - \mmu(\aalpha ') \|_1 +  \| \ppsi_*(\aalpha ') \|_\infty \| \nnu(\aalpha) - \nnu(\aalpha ') \|_1.
    \end{align}
    Since the right hand sides of \eqref{eq: proof 3.1 3} and \eqref{eq: proof 3.1 4} are bounded from above by 

    \[
    \max\big( \| \pphi_*(\aalpha) \|_\infty , \| \pphi_*(\aalpha ') \|_\infty  \big) \| \mmu(\aalpha) - \mmu(\aalpha ') \|_1 
        +  \max\big( \| \ppsi_*(\aalpha) \|_\infty , \| \ppsi_*(\aalpha' ) \|_\infty  \big) \| \nnu(\aalpha) - \nnu(\aalpha ') \|_1,    
    \]%
    we get the desired result.
\end{proof}
Notice that the bound given in \eqref{prop: alpha dependency - bound with potentials} depends on one hand on the $\ell_1$ norm of the difference of parametrized marginals, and we prove in the following lemma that this difference is bounded by a term depending continuously on $\aalpha$. On the other hand, it depends on the $\ell_\infty$ norms of Kantorovitch potentials, and we explain in the sequel how to bound it using $c$-transforms.
\begin{lem} \label{lemma: bound on the marginals}
    Let $\aalpha  \in \cA$ and $\aalpha ' \in \cA$ be parameters. Then
    \begin{equation}
        \| \mmu(\aalpha) - \mmu(\aalpha ') \|_1 \leq K_x \|\aalpha - \aalpha ' \|_\infty , 
    \end{equation}
    and
        \begin{equation}
        \| \nnu(\aalpha) - \nnu(\aalpha ') \|_1 \leq K_y \|\aalpha - \aalpha ' \|_\infty 
    \end{equation}
\end{lem}
\begin{proof}
    By definition, 
    \[
        \mmu(\aalpha) - \mmu(\aalpha ') = \sum_{k=1}^{K_x} (\aalpha^x_{k} - \aalpha '^{x}_{k}) \mmu_{k}.
    \]
    Then, since the $\mmu_{k}$ are probability measures for all $1\leq k \leq K_x$, we get 
    \[\| \mmu(\aalpha) - \mmu(\aalpha ') \|_1  \leq \|\aalpha^x - \aalpha'^x \|_\infty \sum_{k=1}^{K_x} \| \mmu_{k} \|_1 \leq K_x \|\aalpha - \aalpha' \|_\infty.\]
    %

  
  The bound on $ \| \nnu(\aalpha) - \nnu(\aalpha ') \|_1 $ is obtained analogously.
\end{proof}
In what follows, we use  $c$-transforms (see e.g. Chapter 1 of \cite{santambrogio2015optimal})   to build potentials $(\pphi(\aalpha), \ppsi(\aalpha))$ that are solution of \eqref{eq: dual}, and whose norms are bounded by a constant that does not depend on $\aalpha$. Let us first recall the definition of a $c$-transform in the discrete setting:
\begin{defi}\label{def: $c$-transform}
    The discrete $c$-transform of a vector $ \pphi \in \RR^{\cN_x} $ is the vector $ \pphi^c \in \RR^{\cN_y}$ defined for every $1 \leq j \leq \cN_y$ by :
    \begin{equation}
        \pphi^c_j := \min_{1 \leq i \leq \cN_x} \CC_{i,j} - \pphi_i.
    \end{equation}
    The discrete $\bar c$-transform of a vector  $ \ppsi \in \RR^{\cN_y} $ is the vector $ \ppsi^{\bar{c}} \in \RR^{\cN_x}$ defined for every $1 \leq i \leq \cN_x$ by :
    \begin{equation}
        \ppsi^{\bar c}_i := \min_{1 \leq j \leq \cN_y} \CC_{i,j} - \ppsi_j.
    \end{equation}
\end{defi}
One interest of the $c$-transform is that, given an admissible pair $(\pphi(\aalpha), \ppsi(\aalpha))$ of  \eqref{eq: dual}, the pair $(\pphi(\aalpha), \pphi^c(\aalpha))$ is still admissible and increases the cost (the same holds for $\bar c$-transforms). Thus, given a solution $(\pphi_*(\aalpha), \ppsi_*(\aalpha))$, one can assume that it is virtually improved by $c$-transform and $\bar c$-transform. Moreover, $c$-transforms share the same modulus of continuity of the cost function, as we show in the following.
\begin{lem} \label{lemma: $c$-transform continuity}
    Let $ \pphi \in \RR^{\cN_x} $. Then, denoting by $\pphi^c = (\pphi_j^c)_{1 \leq j \leq \cN_y}$ the $c$-transform of $\pphi$, it holds that, for $1 \leq j, j' \leq \cN_y$,
    \begin{equation}
        | \pphi^c_j - \pphi^c_{j'} | \leq 2 \| \CC \|_\infty.
    \end{equation}
\end{lem}
\begin{proof}
    Let $1 \leq j, j' \leq \cN_y$, then for every $1 \leq i \leq \cN_x$, 
    \begin{equation}
        \CC_{ij} - \pphi_i^c \leq \CC_{ij'} - \pphi_i + 2 \| \CC \|_\infty.
    \end{equation}
    passing to the minimum in $i$ in the left hand side of the inequality, and then in the right hand side, we obtain :
    \begin{equation}
        \pphi^c_j \leq \pphi^c_{j'} + 2 \| \CC\|_\infty,
    \end{equation}
    and by interverting the roles of $j$ and $j'$, we obtain the desired result.
\end{proof}
\begin{rmk} \label{rmk : bar $c$-transform continuity}
    We also have that, for $ \ppsi \in \RR^{\cN_y} $, for $1 \leq i, i' \leq \cN_x$,
    \begin{equation}
        | \ppsi^c_i - \ppsi^c_{i'} | \leq 2 \| \CC \|_\infty.
    \end{equation}
\end{rmk}
\begin{cor} \label{cor: bound on potential}
    For any parameter $\aalpha \in \cA$, there exists a solution $(\pphi_*(\aalpha), \ppsi_*(\aalpha))$ satisfying
    \begin{equation}
        0\leq \pphi_*(\aalpha) \leq 2 \| \CC\|_\infty, 
    \end{equation}
    and
    \begin{equation}
        - 3\| \CC \|_\infty \leq \ppsi_*(\aalpha) \leq  \| \CC \|_\infty.
    \end{equation}
    In particular, $\|\pphi_*(\aalpha)\|_\infty \leq 2\| \CC\|_\infty$ and $\|\ppsi_*(\aalpha)\|_\infty \leq 3\| \CC\|_\infty$.
\end{cor}
\begin{proof}
    Let $(\pphi_*(\aalpha), \ppsi_*(\aalpha))$ be a solution of \eqref{eq: dual}. Without loss of generality, we can improve this pair via $c$-transforms, so that they satisfy the continuity bound given in lemma \ref{lemma: $c$-transform continuity} and remark \ref{rmk : bar $c$-transform continuity}. Moreover, since for every constant $\lambda \in \RR$ , $(\pphi_*(\aalpha) - \lambda, \ppsi_*(\aalpha) + \lambda)$ is still an optimal pair, we can suppose that  $\displaystyle  \min_{1\leq i \leq \cN_x} \pphi_*(\aalpha)_i = 0$. We thus obtain that for $1 \leq i \leq \cN_x $, 
    \begin{equation}
        |\pphi_*(\aalpha)_i - 0|\leq 2 \| \CC \|_\infty.
    \end{equation}
    And since $\pphi_*(\aalpha) \geq 0$, we obtain the bound on $\pphi_*(\aalpha)$ :
    \begin{equation*}
        0 \leq \pphi_*(\aalpha) \leq 2 \| \CC \|_\infty.
    \end{equation*}
    In order to bound $\ppsi_*(\aalpha)$, we use the fact that it is defined as the $\bar c$-transform of $\pphi_*(\aalpha)$, that is, for every $1 \leq j \leq \cN_y$, 
    \begin{equation*}
        \ppsi_*(\aalpha)_j = \min_{1\leq i \leq \cN_x} \CC_{ij} - \pphi_*(\aalpha)_i.
    \end{equation*}
   and we deduce from the bound on $\pphi_*(\aalpha)$:
   \begin{equation*}
       \min_{1\leq i \leq \cN_x} \CC_{ij} - 2 \| \CC \|_\infty \leq \ppsi_*(\aalpha)_j \leq \min_{1\leq i \leq \cN_x} \CC_{ij}.
   \end{equation*}
    We deduce
    \begin{equation*}
        - 3 \| \CC \|_\infty \leq \ppsi_*(\aalpha)_j \leq \| \CC \|_\infty.
   \end{equation*}
\end{proof}
Now, we aggregate the previous results to show that $J(\aalpha)$ depends continuously on the parameter $\aalpha$.

\begin{rmk}\label{rmk: bound on potential}
    In the proof of Corollary \ref{cor: bound on potential}, the roles of $\pphi_*(\alpha)$ and $\ppsi_*(\aalpha)$ can be inverted. Thus we can also conclude that for any parameter $\aalpha \in \cA$, there exists a solution $(\pphi_*(\aalpha), \ppsi_*(\aalpha))$ satisfying 
    \begin{equation}
        0\leq \ppsi_*(\aalpha) \leq 2 \| \CC\|_\infty, 
    \end{equation}
    and
    \begin{equation}
        - 3\| \CC \|_\infty \leq \pphi_*(\aalpha) \leq  \| \CC \|_\infty.
    \end{equation}

\end{rmk}
\begin{prop} \label{prop: contiuity with respect to alpha}
    Let $\aalpha \in \cA$ and $\aalpha' \in \cA$ be two parameters. Then 
    \begin{equation*}
        |J(\aalpha) - J(\aalpha')| \leq \|\CC\|_\infty \left( 2 \max(K_x, K_y)+ 3 \min(K_x + K_y) \right) \|\aalpha  - \aalpha '\|_\infty .
    \end{equation*}
\end{prop}
\begin{proof}
    Let $(\pphi_*(\aalpha), \ppsi_*(\aalpha))$ and $(\pphi_*(\aalpha '), \ppsi_*(\aalpha '))$  be a pair of potential solutions to \ref{eq: dual} at the parameter $\aalpha$. They satisfy the bound of property \ref{prop: alpha dependency - bound with potentials}:
 \begin{equation*}
        |J(\aalpha) - J(\aalpha')| \leq \max\big( \| \pphi_*(\aalpha) \|_\infty , \| \pphi_*(\aalpha ') \|_\infty  \big) \| \mmu(\aalpha) - \mmu(\aalpha ') \|_1 
        +  \max\big( \| \ppsi_*(\aalpha) \|_\infty , \| \ppsi_*(\aalpha' ) \|_\infty  \big) \| \nnu(\aalpha) - \nnu(\aalpha ') \|_1
    \end{equation*}
    We can suppose without loss of generality that both pairs of potentials satisfy the conditions of corollary \ref{cor: bound on potential}, and thus
     \begin{equation*}
        |J(\aalpha) - J(\aalpha')| \leq  2 \| \CC \|_\infty \| \mmu(\aalpha) - \mmu(\aalpha ') \|_1 
        + 3 \| \CC \|\infty \| \nnu(\aalpha) - \nnu(\aalpha ') \|_1.
    \end{equation*}
    Using the continuous dependency of the differences marginals built in lemma \ref{lemma: bound on the marginals}, we obtain
    \begin{equation*}\label{eq: proof kxky1}
        |J(\aalpha) - J(\aalpha')| \leq \|\CC\|_\infty \left( 2 K_x+ 3  K_y \right) \|\aalpha  - \aalpha '\|_\infty.
    \end{equation*}
    Now, we can also find a pair $\left( \pphi_*(\aalpha), \ppsi_*(\aalpha) \right)$ satisfying Remark \ref{rmk: bound on potential}. Again, using Lemma \ref{lemma: bound on the marginals}, we obtain 
    \begin{equation*}\label{eq: proof kxky2}
        |J(\aalpha) - J(\aalpha')| \leq \|\CC\|_\infty \left( 2 K_y+ 3  K_x \right) \|\aalpha  - \aalpha '\|_\infty.
    \end{equation*}
    Taking the minimum of the right hand
\end{proof}
Now, our aim is to solve the problem \eqref{eq: primal} for many parameters $\aalpha \in \cA$. However, in high dimensions, solving \eqref{eq: primal} for only one parameter may already be computationally onerous. In the following section, we build a problem of smaller dimensions that approximates \eqref{eq: primal}, namely the reduced-order model, which enables us to approximate solutions of \eqref{eq: primal} at a computational cost that does not depend on the dimensions of the problem \eqref{eq: primal}.

\section{Reduced-order model for the discrete optimal transport problem and its dual}\label{section: RP}
In this section, we introduce  the reduced-order linear 
programs that approximate the problems \eqref{eq: primal} and 
\eqref{eq: dual}.  It consists in building a non-negative cone 
(respectively two vector spaces) of smaller dimension by using 
snapshots methods on the primal problem \eqref{eq: primal} 
(respectively the dual \eqref{eq: dual}), and to rewrite the 
problem where the minimization is performed on these cone or spaces.

We begin by explaining how these cone and spaces are defined. We 
then build the reduced-order model, which reads as a linear program 
with a small number of variables and constraints.
We then provide sufficient conditions under which the reduced-order model has a least one solution. Finally, we propose an 
alternative method to build subspaces $\cU$ and $\cV$ of small 
dimension. 

\subsection{Construction of the reduced bases}
For the approximation of the primal solutions is selected in a reduced non-negative cone $\mathcal{W}_+ \subset \RR_+^{\cN_x \cN_y}
$ spanned by a family of vectors $\left\lbrace \ww_r \right\rbrace_{1 \leq r \leq R} \subset \RR_+^{\cN_x \cN_y}$ for $R \ll \cN_x  \cN_y$. Concerning  the dual solutions we approximate them by looking at the elements of reduced subspaces $\mathcal{U} \subset \RR^{\cN_x}$ and $\mathcal{V} \subset \RR^{\cN_y}$ spanned by two families of vectors $\left\lbrace \uu_n \right\rbrace_{1 \leq n \leq N} \subset \RR^{\cN_x}$  and  $\left\lbrace \vv_m \right\rbrace_{1 \leq m \leq M} \subset \RR^{\cN_y}$ for $N \ll \cN_x $ and $M \ll \cN_y$ .

The reduced variables will then be written as :
\begin{equation}\label{eq: reduced solutions}
\begin{array}{ll}
    \displaystyle \widehat{\ppi}(\aalpha) = \sum_{r=1}^R \pp_r(\aalpha) \ww_r \; \text{with } \pp(\aalpha)\in \RR^R_+ \\
    \displaystyle  \widehat{\pphi}(\aalpha) = \sum_{n=1}^N \aa_n(\aalpha) \uu_n \; \text{with }  \aa(\aalpha) \in \RR^N \\
    \displaystyle  \widehat{\ppsi}(\aalpha) = \sum_{m=1}^M \bb_m(\aalpha) \vv_m \; \text{with }  \bb(\aalpha) \in \RR^M \\
\end{array}
\end{equation}
In the following, we will denote by $\widehat{\WW}$ the matrix whose columns are the vectors $\left\lbrace \ww_r \right\rbrace_{1 \leq r \leq R}$, and by $\widehat{\UU}$ (respectively $\widehat{\VV}$) the matrix composed by $\left\lbrace \uu_n \right\rbrace_{1\leq n \leq N}$ (resp.
by  $\left\lbrace \vv_m \right\rbrace_{1 \leq m \leq M}$).

Using these notations, the reduced solution may be rewritten as matrix-vector multiplication:
\begin{equation}\label{eq: reduced solutions matrix}
\begin{array}{ll}
    \widehat{\ppi}(\aalpha) = \widehat{\WW} \pp(\aalpha) \\
    \widehat{\pphi}(\aalpha) = \widehat{\UU} \aa(\aalpha) \\
    \widehat{\ppsi}(\aalpha) = \widehat{\VV} \bb(\aalpha) \\
\end{array}
\end{equation}
\paragraph{The Primal reduced space \texorpdfstring{$\mathcal{W}_+$}{}} \label{subsubsection : primal reduced space W}
We admit here that the set of parameters $\cA$ is a convex set generated by a finite number of extreme points. Let $\cA_{\text{ext}} \subset \cA$ be the set of extreme points of $\cA$.
Let  $\cA_{\text{train}}^P$ be a finite training set  containing the set of extreme points, namely $\cA_{\text{ext}} \subset \cA_{\text{train}}^P \subset \cA$.
 Then $\cW_+ $ is defined as the cone generated by the solutions to the problem \eqref{eq: primal} for all the parameters $\aalpha \in \cA_{\text{train}}^P$, that is 
 \begin{equation}
      \cW_+ := \text{cone}_+ \left\lbrace {\ppi}_*(\aalpha) \right\rbrace_{\alpha \in \mathcal{A}_{\text{train}}^P} ,
 \end{equation}
and $\widehat{\WW}$ is the matrix formed by stacking the transport plans $\left\lbrace \ppi_*(\aalpha) \right\rbrace_{\alpha \in \mathcal{A}_{\text{train}}^P}$ columnwise.
We explain in Section \ref{subsection: non emptyness} that the  inclusion $\cA_{\text{ext}} \subset \cA_{\text{train}}^P$ guarantees the non-emptiness of the feasible set.

Note that, as usual in basis reduction,  it is possible to reduce the dimension of $\cW_+$ with a non-negative matrix factorization. In this case, the solutions at the extreme points $\cA_{\text{ext}}$ must be added to $\widehat{\WW}$.
%
\paragraph{The Dual reduced spaces \texorpdfstring{$\mathcal{U}$ and $\mathcal{V}$}{} }
Let us assume that we are given a data set of dual solutions $\left\lbrace (\pphi_*(\aalpha), \ppsi_*(\aalpha) ) \right\rbrace_{\aalpha \in \cA_{\text{train}}^D}$ for a finite training subset $\cA_{\text{train}}^D \subset \cA$. 
The dual reduced space $\cU$ (resp. $\cV$) is the space generated by $\left\lbrace (\pphi_*(\aalpha) \right\rbrace_{\aalpha \in \cA_{\text{train}}^D}$ and $\11_{\cN_x}$ (resp. by $\left\lbrace (\ppsi_*(\aalpha) \right\rbrace_{\aalpha \in \cA_{\text{train}}^D}$ and $\11_{\cN_y}$ ). We explain in Section \ref{subsection: compactness} why the $\11_{\bullet}$ vectors are added. We propose later in Section \ref{subsection : small U and V} an alternative way of building $\cU$ and $\cV$, as the corresponding method requires further understanding of the reduced-order model.

\subsection{The reduced-order model}
We now explain how to use the reduced cones and bases in order to obtain a reduced-order model. In particular we  proceed in two steps: (1) we  build a semi-reduced-order model, where the ambient space  of the high-fidelity primal \eqref{eq: primal} is reduced (2) we further reduce the dimension of the space variable of the dual of this semi-reduced-order model.

Notice that by inverting the two steps (we first reduce the dimension of the variables of the dual and then reduce the dimension of the variables of the associated primal) we obtain in the end the same fully reduced-order model, so only the semi-reduced-order models will differ. We explain how both of them are obtained in what follows.
\paragraph{Reduction on Primal-then-Dual}
We propose here a semi-reduced-order model where the variables of the primal are searched in a non-negative cone of smaller dimension. It is obtained as rewriting \eqref{eq: primal} with the added constraint that the solutions should belong to the non-negative cone $\cW_+$:
\begin{equation}\label{eq: RPP1}
I_{R_P}(\aalpha):= \mathop{\inf}_{\begin{array}{c}
\widehat \ppi\in \cW_+\\
\AA \widehat\ppi= \begin{pmatrix} \mmu(\aalpha) \\ \nnu(\aalpha)\end{pmatrix}
\end{array}
}
\left\langle \widehat\ppi , \cc \right\rangle .
\end{equation}
If we use the notations introduced previously \eqref{eq: reduced solutions matrix}, this problem can be reformulated in the following form :
\begin{equation}\label{eq: RPP}\tag{$R_PP_{\aalpha}$}
I_{R_P}(\aalpha):= \mathop{\inf}_{\begin{array}{c}
\pp \in \RR_+^R\\
\AA \widehat{\WW} \pp = \begin{pmatrix} \mmu(\aalpha) \\ \nnu(\aalpha)\end{pmatrix}
\end{array}
}
\left\langle \pp, \widehat{\WW}^T\cc \right\rangle .
\end{equation}
The dual of this semi-reduced-order problem writes
\begin{equation}\label{eq: RPD}\tag{$R_PD_{\aalpha}$}
{J_{R_P}(\aalpha) = \mathop{\sup}_{ \begin{array}{c}
(\pphi,\ppsi) \in \RR^{\cN_x + \cN_y} \\
\widehat\WW^T \AA^T \begin{pmatrix} \pphi \\ \ppsi \end{pmatrix} \leq \widehat\WW^T \cc \\
\end{array}
} 
\left\langle \pphi, \mmu(\aalpha) \right\rangle + \left\langle \ppsi, \nnu(\aalpha) \right\rangle .}
\end{equation}
We now add to this semi-reduced-order dual problem the constraint that the Kantorovitch potentials should belong to the subspaces $\cU$ and $\cV$. We obtain the following reduced dual:
\begin{equation}\label{eq: RD1}
\boxed{\begin{array}{lll}
     J_{R}(\aalpha) &:=&\displaystyle \mathop{\sup}_{ \begin{array}{c}
(\widehat{\pphi}, \widehat{\ppsi}) \in \cU \times \cV \\
\widehat\WW^T \AA^T \begin{pmatrix} \widehat{\pphi} \\ \widehat{\ppsi} \end{pmatrix} \leq \widehat\WW^T \cc \\
\end{array}
} 
\left\langle \widehat{\pphi} , \mmu(\aalpha) \right\rangle + \left\langle \widehat{\ppsi}, \nnu(\aalpha) \right\rangle
 \\
& =&  \displaystyle\mathop{\sup}_{ \begin{array}{c}
(\aa,\bb) \in \RR^{N+M} \\
\widehat\WW^T \AA^T  \widehat\sS \begin{pmatrix} \aa\\ \bb \end{pmatrix} \leq \widehat\WW^T \cc \\
\end{array}
} 
\left\langle \aa, \widehat\UU^T\mmu(\aalpha) \right\rangle + \left\langle \bb, \widehat\VV^T \nnu(\aalpha) \right\rangle,
\end{array}}
\end{equation}
where
\begin{equation}\label{eq: S}
    \widehat{\sS} := \begin{pmatrix}
   \widehat{\UU} & \mathbb{0} \\ \mathbb{0}  & \widehat{\VV}
\end{pmatrix}.
\end{equation}
The primal reduced problem then writes 
\begin{equation}\label{eq: RP proof}
\boxed{\begin{array}{lll}
I_R(\aalpha) &:=& \displaystyle \mathop{\inf}_{\begin{array}{c}
\widehat \ppi\in \mathcal{W}_+ \\ 
\widehat{\sS}^T \AA \widehat\ppi= \widehat{\sS}^T\begin{pmatrix} {\mmu}(\aalpha) \\ \nnu(\aalpha)\end{pmatrix}
\end{array}
} \left\langle \widehat\ppi, {\cc} \right\rangle, \\
&=& \displaystyle \mathop{\inf}_{\begin{array}{c}
 \pp \in \RR_+^R \\ 
\widehat\sS^T \AA \widehat\WW \pp = \widehat \sS^T \begin{pmatrix} \mmu(\aalpha) \\ \nnu(\aalpha)\end{pmatrix}
\end{array}
} \left\langle  \pp, \widehat\WW^T \cc \right\rangle.
\end{array}}
\end{equation}
We can re-write this problem in the standard form of linear programming : let $\widehat{\cc}:= \widehat{\textbf{W}}^T \cc $, $\widehat{\mmu}(\aalpha) := \widehat{\UU}^T\mmu(\aalpha)$ and $\widehat{\nnu}(\aalpha) := \widehat{\VV}^T\nnu(\aalpha)$, and let $\widehat{\AA} := \widehat{\sS}^T \AA  \widehat{\textbf{W}}$.

The reduced problem then  writes

\begin{equation}\label{eq: RP}\tag{$RP_{\aalpha}$}
\boxed{I_R(\aalpha):= \mathop{\inf}_{\begin{array}{c}
\pp \in \RR_+^R \\ 
\widehat{\AA} \pp = \begin{pmatrix} \widehat{\mmu}(\aalpha) \\ \widehat\nnu(\aalpha)\end{pmatrix}
\end{array}
} \left\langle \pp, \widehat{\cc} \right\rangle,} 
\end{equation}

and its dual

\begin{equation}\label{eq: RD}\tag{$RD_{\aalpha}$}
\boxed{J_R(\aalpha):= \mathop{\inf}_{\begin{array}{l}
\aa \in \RR^N, \bb \in \RR^M \\ 
\widehat{\AA}^T \begin{pmatrix} \aa \\ \bb\end{pmatrix} \leq \widehat\cc
\end{array}
} \left\langle \aa, \widehat\mmu(\aalpha) \right\rangle + \left\langle \bb, \widehat\nnu(\aalpha) \right\rangle,} 
\end{equation}
For convenience in the proofs of the properties that follow, we will sometimes refer to the form \eqref{eq: RP proof} of the reduced-order model.

The dimensions of the high-fidelity model \eqref{eq: primal} and the reduced-order model are compared in table \ref{tab: dimensions}.
\begin{table}[htbp]
\centering
\begin{tabular}{l c c}
\toprule
Problem & Degrees of Freedom & Constraints \\
\midrule
High-Fidelity & $\cN_x \cN_y$ & $\cN_x + \cN_y$ \\
Redced-order Model & $R$ & $N+M$ \\
\bottomrule
\end{tabular}
\caption{Dimensions comparison.}\label{tab: dimensions}
\end{table}
\paragraph{Reduction on dual-then-primal}
Here we reduce the problem by first reducing the dimensions of the variables of the dual high-fidelity problem \eqref{eq: dual}, and then obtaining the reduced-order model by reducing again on the dimensions of the space variables of the primal problem. So we begin by adding to the high-fidelity dual problem the constraint that the Kantorovitch potential should belong to the subspaces $\cU$ and $\cV$:

\begin{equation}\label{eq: RDD1}
  J_{R_D}(\aalpha) = \mathop{\sup}_{ \begin{array}{c}
(\widehat\pphi,\widehat\ppsi) \in \cU\times \cV \\
\AA^T \begin{pmatrix} \widehat\pphi \\ \widehat\ppsi \end{pmatrix} \leq  \cc \\
\end{array}
} 
\left\langle \widehat\pphi, \mmu(\aalpha) \right\rangle + \left\langle \widehat \ppsi, \nnu(\aalpha) \right\rangle .
\end{equation}
Using the notations introduced in \ref{eq: reduced solutions matrix} and in \ref{eq: S}, we rewrite this problem with reduced dimensions
\begin{equation}\label{eq: RDD}\tag{$R_DD_{\aalpha}$}
J_{R_D}(\aalpha) = \mathop{\sup}_{ \begin{array}{c}
(\aa,\bb) \in \RR^{N + M} \\
 \AA^T \widehat \sS \begin{pmatrix} \aa \\ \bb \end{pmatrix} \leq  \cc \\
\end{array}
} 
\left\langle \aa, \UU^T \mmu(\aalpha) \right\rangle + \left\langle \bb, \VV^T \nnu(\aalpha) \right\rangle .
\end{equation}
The primal version of this linear program writes
\begin{equation}\label{eq: RDP}\tag{$R_DP_{\aalpha}$}
I_{R_D}(\aalpha):= \mathop{\inf}_{\begin{array}{c}
\ppi\in \RR_+^{\cN_x\cN_y} \\ 
\widehat\sS^T \AA \ppi= \widehat \sS^T \begin{pmatrix} \mmu(\aalpha) \\ \nnu(\aalpha)\end{pmatrix}
\end{array}
} \left\langle \ppi,  \cc \right\rangle, 
\end{equation}
We now reduce the dimension of the cone of variables :
\begin{equation}\label{eq: RP2
}
I_{R}(\aalpha):= \mathop{\inf}_{\begin{array}{c}
\widehat \ppi\in \cW+ \\ 
\widehat\sS^T \AA \widehat\ppi= \widehat \sS^T \begin{pmatrix} \mmu(\aalpha) \\ \nnu(\aalpha)\end{pmatrix}
\end{array}
} \left\langle \widehat\ppi,  \cc \right\rangle, 
\end{equation}
and we recognize the reduced-order model \eqref{eq: RP proof}.

\subsection{Existence of solution for the reduced-order model}\label{subsection: non emptyness}
The high-fidelity problem \eqref{eq: primal} reads as a minimization problem of a continuous function in an non empty compact set, as the vector $\mmu(\aalpha)\otimes\nnu(\aalpha)$ is admissible and the feasible set is a closed subset of the probability vectors of dimension $\cN_x\cN_y$. Therefore, it has solutions.
In the reduced-order model, the function to minimize is continuous, but the conditions on the feasible set are a priori not satisfied anymore. We propose in what follows a condition on the cone $\cW_+$ to ensure that the feasible set is not empty, and a condition  on the matrices $\widehat\UU$ and $\widehat\VV$ to guarantee compactness. 

\paragraph{Non emptiness of the feasible set}\label{subsubsection : non emptyness of the feasible set}
We first exhibit that  $\cA$ is a convex polytope, thus can be generated by finitely many extreme points, and so does the family of marginals $\left\lbrace \begin{pmatrix} \mmu(\aalpha) \\ \nnu(\aalpha) \end{pmatrix} \right\rbrace_{\aalpha \in \cA}$. We then deduce that the reduced-order model \eqref{eq: RP} is feasible for any $\aalpha \in \cA$ if the reduced cone $\cW_+$ contains a transport plan solution of the discrete high-fidelity problem \eqref{eq: primal} at each  extreme point of $\cA$. 
\begin{prop}
The set $\cA$ is a finite intersection of half-spaces, and it is bounded. In other words, it is a convex polytope.
\end{prop}
\begin{cor}\label{cor: finite number of extreme points}
  $\cA$ is generated by the set of its extreme points, denoted $\cA_{\text{ext}}$, and this set is finite.
\end{cor}
This corollary is a classical property of the analysis of convex polytopes. For a proof we refer for instance the reader to Chapter 1 of \cite{karloff2008linear}. We now show that the set $\left\lbrace \begin{pmatrix} \mmu(\aalpha) \\ \nnu(\aalpha) \end{pmatrix} \right\rbrace_{\aalpha \in \cA}$ inherits the same structure as $\cA$.
\begin{lem}\label{lemma: extreme marges}
  Let $\aalpha \in \cA$ and its convex decomposition with extreme points be $\displaystyle \aalpha = \sum_{\aalpha_{\text{ext}} \in \cA_{\text{ext}}} \lambda_{\text{ext}} \aalpha_{\text{ext}}$. Then
  \begin{equation}
      \begin{pmatrix}
          \mmu(\aalpha) \\ \nnu(\aalpha) 
      \end{pmatrix} = 
      \sum_{\aalpha_{\text{ext}} \in \cA_{\text{ext}}} \lambda_{\text{ext}} \begin{pmatrix}
          \mmu(\aalpha_{\text{ext}}) \\ \nnu(\aalpha_{\text{ext}}) 
      \end{pmatrix} 
  \end{equation}
\end{lem}
\begin{proof}
Let us remark that the set of parametrized marginals $\left\lbrace \begin{pmatrix} \mmu(\aalpha) \\ \nnu(\aalpha) \end{pmatrix} \right\rbrace_{\aalpha \in \cA}$ can be seen as the image of $\cA$ by the linear application 
$\textbf{M} := \begin{pmatrix}
    \mmu_1 & \dots & \mmu_{K_x} & \00 & \dots & \00 \\
    \00 & \dots & \00 & \nnu_1 & \dots & \nnu_{K_y}
\end{pmatrix}$ .
Now, let $\aalpha \in \cA$. With the previous notations, 
\begin{equation}
    \begin{pmatrix}
        \mmu(\aalpha) \\ \nnu(\aalpha) 
    \end{pmatrix} = \textbf{M} \aalpha.
\end{equation}
But since $\displaystyle \aalpha = \sum_{\aalpha_{\text{ext}} \in \cA_{\text{ext}}} \lambda_{\text{ext}} \aalpha_{\text{ext}}$, we also have
\begin{equation}
    \begin{pmatrix}
        \mmu(\aalpha) \\ \nnu(\aalpha) 
    \end{pmatrix}
    = \sum_{\aalpha_{\text{ext}} \in \cA_{\text{ext}}} \lambda_{\text{ext}} \textbf{M} \aalpha_{\text{ext}} 
     = \sum_{\aalpha_{\text{ext}} \in \cA_{\text{ext}}} \lambda_{\text{ext}} \begin{pmatrix}
          \mmu(\aalpha_{\text{ext}}) \\ \nnu(\aalpha_{\text{ext}}) 
      \end{pmatrix} .
\end{equation}
\end{proof}
We are now able to give a feasibility condition.
\begin{prop}\label{prop: feasible set not empty}
    Suppose that for any extreme point of $\cA$ $\aalpha_{\text{ext}}$, the solution $\ppi(\aalpha_{\text{ext}})$ of the high-fidelity discrete problem \ref{eq: primal} belongs to $\cW_+$. Then, for any $\aalpha \in \cA$, the problem \ref{eq: RP} is feasible.
\end{prop}
\begin{proof}
    Let $\aalpha \in \cA$ be a given parameter, and let $ \displaystyle \aalpha = \sum_{\aalpha_{\text{ext}} \in \cA_{\text{ext}}} \lambda_{\text{ext}} \aalpha_{\text{ext}}$ be its convex decomposition over the extreme points of $\cA$. Let $\displaystyle  \ppi= \sum_{\aalpha_{\text{ext}} \in \cA_{\text{ext}}} \lambda_{\text{ext}} \ppi(\aalpha_{\text{ext}})$. Our goal is to prove that $\ppi$ is feasible, that is $\ppi\in \cW_+$, and. $\widehat\sS^T \AA \ppi= \widehat \sS^T\begin{pmatrix} \mmu(\aalpha) \\ \nnu(\aalpha) \end{pmatrix}$ (we refer to the feasibility conditions baed on the formulation \eqref{eq: RP proof} of the reduced order model). Let us first notice that by hypothesis, $\ppi$ is a convex (thus positive) combination of points of $\cW_+$, and thus $\ppi\in \cW_+$. It remains to prove that $\ppi$ satisfies the constraints. One one hand, we can rewrite the left hand side of the constraint equality : 
    \begin{equation}
        \AA \ppi= \sum_{\aalpha_{\text{ext}} \in \cA_{\text{ext}}} \lambda_{\text{ext}} \AA \ppi(\aalpha_{\text{ext}}) = \sum_{\aalpha_{\text{ext}} \in \cA_{\text{ext}}} \lambda_{\text{ext}} \begin{pmatrix}\mmu(\aalpha_{\text{ext}}) \\ \nnu(\aalpha_{\text{ext}}) \end{pmatrix},
    \end{equation}
    and on the other hand, by lemma \ref{lemma: extreme points}, we rewrite the right-hand side of the constraint equality
    \begin{equation}
         \begin{pmatrix}\mmu(\aalpha) \\ \nnu(\aalpha)\end{pmatrix} = \sum_{\aalpha_{\text{ext}} \in \cA_{\text{ext}}} \lambda_{\text{ext}} \begin{pmatrix}\mmu(\aalpha_{\text{ext}}) \\ \nnu(\aalpha_{\text{ext}}) \end{pmatrix} ,
    \end{equation}
    thus $\ppi$ is admissible and the feasible set of \eqref{eq: RP proof} is not empty.
\end{proof}
\begin{rmk}
    Notice  that if the conditions of Proposition \ref{prop: feasible set not empty} holds, then the semi-reduced problem \eqref{eq: RPP} is feasible too.
\end{rmk}
\begin{rmk} \label{rmk : non emptyness = convex combination of columns.}
    Since the matrix $\widehat \WW$ is obtained by stacking columnwise the transport plans $\{ \ppi_*(\aalpha)\}_{\aalpha \in \cA_\text{train}^P}$, $\AA \widehat \WW $ is obtained by stacking columnwise the corresponding constraints $\left\lbrace \begin{pmatrix} \mmu(\aalpha) \\ \nnu(\aalpha)\end{pmatrix}\right\rbrace_{\aalpha \in \cA_\text{train}^P}$. Saying that the feasible set of the semi-reduced problem is not empty at $\aalpha \in \cA$ is equivalent to the fact that there exists a positive combination of columns of $\AA \widehat \WW$ 
    that is equal to $\begin{pmatrix} \mmu(\aalpha) \\ \nnu(\aalpha)\end{pmatrix}$.

\end{rmk}
We finish by giving an explicit expression of the extreme points of $\cA$, since it can be convenient for numerical applications.
\begin{lem}\label{lemma: extreme points}
    The set $\displaystyle  \cA := \left\lbrace \aalpha := ({\aalpha^x, \aalpha^y})\in \mathbb{R}_+^{K_x + K_y}, \sum_{k=1}^{K_x} \mathbf{\aalpha}^x_k = 1, \sum_{l=1}^{K_y} \mathbf{\aalpha}^y_l = 1 \right\rbrace$  has  $K_x\times K_y$ extreme points, defined by $\cA_{\text{ext}} = \left\lbrace (\ee_{k}, \ff_{l})\right\rbrace_{\footnotesize \begin{matrix} 1 \leq k \leq K_x \\ 1\leq l \leq K_y \end{matrix}}$, where $\left\lbrace \ee_k \right\rbrace_{1 \leq k \leq K_x}$ and $\left\lbrace \ff_k \right\rbrace_{1 \leq l \leq K_y}$ are the elementary basis vectors of $\RR^{K_x}$ and $\RR^{K_y}$.
\end{lem}
\begin{proof}
We begin by showing the generative property of $\cA_{\text{ext}}$. We then prove that the elements of $\cA_{\text{ext}}$ are extreme points by noticing each point $\aalpha$ of $\cA_{\text{ext}}$, can be expressed as one unique convex combination of points of $\cA_{\text{ext}} $.

Let $(\aalpha^x, \aalpha^y) = \left(\left\lbrace \aalpha_{k}^x \right\rbrace_{1 \leq k \leq K_x}, \left\lbrace \aalpha_{l}^y \right\rbrace_{1\leq l \leq K_y} \right)$ be an element of $\cA$. Let us show that $\aalpha$ is a convex combination of  $\left\lbrace \left( \ee_{k}, \ff_{l}\right) \right\rbrace_{\footnotesize \begin{matrix} 1 \leq k \leq K_x \\ 1\leq l \leq K_y \end{matrix}}$ associated with the coefficients $\bbeta = \left\lbrace \aalpha_{k}^x \aalpha_{l}^y\right\rbrace_{\footnotesize \begin{matrix} 1 \leq k \leq K_x \\ 1\leq l \leq K_y \end{matrix}} $:

\begin{equation}
    \sum_{k' = 1}^{K_x} \sum_{l' = 1}^{K_y} \aalpha_{k'}^x\aalpha_{l'}^y (\ee_{k'}, \ff_{l'})  = (\aalpha^x, \aalpha^y).
\end{equation}

We first check that the coefficients of $\bbeta$ form a convex combination. This arises from the fact that the coefficients of $\aalpha^x$ and $\aalpha^y$ are themselves  convex combinations, so for any $1\leq k' \leq K_x$ and $1 \leq l' \leq K_y$, $\bbeta_{k', l'} \geq 0$, and

\begin{equation}
   \sum_{k' = 1}^{K_x} \sum_{l' = 1}^{K_y} \bbeta_{k', l'} = \sum_{k' = 1}^{K_x} \sum_{l' = 1}^{K_y} \aalpha_{k'}^x\aalpha_{l'}^y  = 1 .
\end{equation}

Now, notice that we can decompose $\aalpha^x$ and $\aalpha^y$ as sums of elementary vectors: $\displaystyle  \aalpha^x = \sum_{k = 1}^{K_x} \aalpha^x_{k} \ee_{k}$, and $\displaystyle  \aalpha^y = \sum_{l = 1}^{K_y} \aalpha^y_{l} \ff_{l}$.

Thus, for $1 \leq k \leq K_x$, 
\begin{equation}
    \begin{array}{ll}
        \displaystyle{} \left(\sum_{k' = 1}^{K_x} \sum_{l' = 1}^{K_y} \aalpha_{k'}^x \aalpha_{l'}^y \ee_{k'}\right)_{k} & = \displaystyle{}
          \left(\sum_{k' = 1}^{K_x} \aalpha_{k'}^x \ee_{k'} \right)_{k}\\
         & \displaystyle{} =  (\aalpha_{k}^x \ee_{k})_{k} \\
         &\displaystyle{} = \aalpha^x_{k},
    \end{array}
\end{equation}

and for $1 \leq l \leq K_y$, 
\begin{equation}
    \begin{array}{ll}
       \displaystyle{}   \left(\sum_{k' = 1}^{K_x} \sum_{l' = 1}^{K_y} \aalpha_{k'}^x \aalpha_{l'}^y  \ff_{l'}\right)_{l} & = 
        \displaystyle{}  \left(\sum_{l' = 1}^{K_y} \aalpha_{l'}^y \ff_{l'} \right)_{l}\\
         &\displaystyle{} = ( \aalpha_{l}^y \ff_{l})_{l} \\
         &\displaystyle{} = \aalpha^y_{l}
    \end{array}
\end{equation}

This concludes the fact that $\left\lbrace \left( \ee_{k}, \ff_{l}\right) \right\rbrace_{\footnotesize \begin{matrix} 1 \leq k \leq K_x \\ 1\leq l \leq K_y \end{matrix}}$ is a generative set of $\cA$.

Now, suppose that, for $1\leq k \leq K_x$ ans $1\leq l \leq K_y$, the point $\displaystyle{} (\ee_{k'}, \ff_{l'})$ is expressed as a convex combination of $\left\lbrace \left( \ee_{k}, \ff_{l}\right) \right\rbrace_{\footnotesize \begin{matrix} 1 \leq k \leq K_x \\ 1\leq l \leq K_y \end{matrix}}$ associated to the coefficients $\{ \bbeta_{k', l'}\}$. Then, in particular :
\begin{equation}
    \ee_{k} = \sum_{k' =1}^{K_x} \left( \sum_{l' =1}^{K_y} \bbeta_{k', l'} \right) \ee_{k'},
\end{equation}
By linear independence of the elementary vectors $\ee_{k'}$, if $k' \neq k$,
 \begin{equation}
     \sum_{l' = 1}^{K_y} \bbeta_{k', l'} = 0,
 \end{equation}
and by non-negativity of the coefficients of $\bbeta$, we obtain that $\bbeta_{k', l' = 0}$ if $k' \neq k$.

By a similar computation, since the $\{\ff_{l'}\}_{1 \leq l'\leq K_y}$ are independent, by non-negativity of the coefficients of $\bbeta$, for every $1\leq k' \leq K_x$, $\leq l' \leq K_y$, $\bbeta_{k', l'} = 0$ if $l' \neq l$. 

Since the coefficients of $\bbeta$ sum to one, this means that the only coefficient that can be non zero, namely $\bbeta_{k,l}$, is equal to 1. So we have proven that every  element of $\left\lbrace \left( \ee_{k}, \ff_{l}\right) \right\rbrace_{\footnotesize \begin{matrix} 1 \leq k \leq K_x \\ 1\leq l \leq K_y \end{matrix}}$ can be written as one unique convex combination of elements of this set.
\end{proof}

\paragraph{Compactness of the feasible set} \label{subsection: compactness}
We draw here a simple condition to guaranty the compactness of the reduced-order model \eqref{eq: RP}.

\begin{lem}\label{lemma: compactness}
If  the vector $\11_{\cN_x}$ is one column of $\widehat{\UU}$ or if $\11_{\cN_y}$ is one column of $\widehat{\VV}$, the set of feasible solutions of  \eqref{eq: RP} is compact.    
\end{lem}
\begin{proof}
The feasible set of the reduced problem written on the form \eqref{eq: RP proof} is closed as it is the intersection between a closed cone, $\mathcal{W}_+$, and the reciprocal image of a point by a linear application, $\mathrm{Im}^{-1}(\widehat{\sS}^T \AA)\left\lbrace \begin{pmatrix} \widehat{\mmu}(\aalpha)\\ \widehat{\nnu}(\aalpha)\end{pmatrix}\right\rbrace $.

If $\11_{\cN_x}$ is a column of $\widehat{\textbf{U}}$, for instance, the last one, then 
\begin{equation}
 \widehat{\sS}^T_N =    (\widehat{\textbf{U}}^T| \00)_N = (\11_{\cN_x}| \00 ) .
\end{equation}
Recall now that, by definition of $\AA$ in \eqref{eq: A}, for any $\ppi \in \RR^{\cN_x \cN_y}_+$, 
\begin{equation}
    (\AA \ppi)_i = \sum_{j = 1}^{\cN_y} \ppi_{i + j \cN_x}.
\end{equation}
Thus, the condition $\left(\widehat{\sS}^T \AA\right)_N \ppi=\left(\widehat{\sS}^T \AA\right)_N \begin{pmatrix} \widehat{\mmu}(\aalpha) \\ \widehat{\nnu}(\aalpha) \end{pmatrix}$ reads explicitly as the following equality:
\begin{equation}
         \sum_{i = 1}^{\cN_x}(\AA \ppi)_i = \sum_{k = 1}^{\cN_x \times \cN_y} \ppi_k  = \sum_{i = 1}^{\cN_x} \mmu(\aalpha)_i = 1.
\end{equation}
Since $\ppi \in \RR^{\cN_x \cN_y}_+$ , the fact that its coefficients sum to 1  implies that the feasible set is bounded.

If $\11_{\cN_y}$ is one column of $\widehat\VV$, the same proof holds with very similar computations. 
\end{proof}
\begin{rmk}
    Adding the $\11_\bullet$ vector is in fact a way to say we are looking for solutions that are probability vectors.

\end{rmk}
\paragraph{Existence of a solution}
We end this Section by giving an existence result for the reduced-order model.
\begin{prop}\label{prop: existence of a solution}
    Under the conditions of Properties \ref{prop: feasible set not empty} and \ref{lemma: compactness}, the reduced-order primal has a solution.  
\end{prop}
\begin{proof}
    The reduced-order primal is a minimization problem of a continuous function. The conditions of \ref{prop: feasible set not empty} and \ref{lemma: compactness} guaranty that the feasible set is not empty and compact. We deduce that there exists a solution to this problem.
\end{proof}
\begin{rmk}
  By strong duality theory (see e.g. Chapter 3 of \cite{karloff2008linear}), it is sufficient to prove that the reduced-order problem \eqref{eq: RP} has solution to guaranty that its dual version \eqref{eq: RD} has solutions too.  
\end{rmk}

\subsection{An alternative way to choose \texorpdfstring{$\cU$ and $\cV$}{}} \label{subsection : small U and V}
We propose here a way to choose $\cU$ and $\cV$ of dimension $K_x$ and $K_y$, that guarantees that the optimal costs of the reduced-order model \eqref{eq: RP} and of the semi-reduced order model \eqref{eq: RPP} are the same. We begin by giving a sufficient condition to obtain the equality of the optimal value. We then present a way to build $\widehat \UU$ and $\widehat \VV$ satisfying this condition with Gram-Schmidt algorithm.

\subsubsection{Rank condition to make the reduced-order model equivalent to the semi-reduced-order model}
Let us begin by giving a short analysis of the structure of the constraint matrix of the half reduced problem \eqref{eq: RPP} that motivates the search of $\widehat \UU$ and $\widehat \VV$ of small dimension. As defined in \ref{subsubsection : primal reduced space W}, $\widehat{\WW}$ is obtained by stacking columnwise solutions of high-fidelity problems for different  parameters. In other terms, if the corresponding training set is defined as $\cA_{train}^P= \{ \aalpha_r\}_{1\leq r \leq R}$, then:
\begin{equation}
    \widehat{\WW} := \begin{pmatrix} \ppi_*(\aalpha_1)| \cdots | \ppi_*(\aalpha_R) \end{pmatrix}.
\end{equation}
%
One can now easily see that the constraint matrix $\AA \widehat \WW$ can be written as 
\begin{equation}
    \AA\widehat{\WW} =   \begin{pmatrix} \mmu(\aalpha_1)| \cdots | \mmu(\aalpha_R) \\ \nnu(\aalpha_1)|\cdots | \nnu(\aalpha_R) \end{pmatrix} =: \begin{pmatrix} (\AA \widehat \WW)_x \\ (\AA \widehat \WW)_y \end{pmatrix}, 
\end{equation}
where  $(\AA\widehat\WW)_x$ (respectively $(\AA\widehat\WW)_y$) denotes the upper part (resp. lower part) of $\AA\widehat\WW$.
We can now derive  an upper bound of the rank of $(\AA\widehat\WW)_x$  and $(\AA\widehat\WW)_y$ :
\begin{lem} \label{lemma: rank of cstr AW}
   $\mathrm{rank}\big((\AA\widehat\WW)_x\big) \leq K_x$ and $\mathrm{rank}\big((\AA\widehat\WW)_y\big) \leq K_y$.  
\end{lem}
\begin{proof}
In this proof, we only focus on $(\AA\widehat\WW)_x$, but the same reasoning can be applied on $(\AA\widehat\WW)_y$.
For every parameter $\aalpha$, the marginal $\mmu(\aalpha)$ is a convex (thus linear) combination of the $K_x$ marginals $\{ \mmu_{k} \}_{1 \leq k \leq K_x}$. In particular, since every column of $(\AA\widehat\WW)_x$ is a linear combination of this family of marginals, whose rank is at most $K_x$.
\end{proof}
This rank results motivates us to find some $\widehat \UU$ and $\widehat \VV$ that would have at most $K_x$ and $K_y$ columns for the reduced problem, and that would however satisfy that the optimal value fo the half reduced problem \eqref{eq: RPP} and the reduced problem \eqref{eq: RP} would be the same.
The following property gives a sufficient condition on $\widehat \UU$ and $\widehat \VV$ to guarantee that these optimal values are equal.
\begin{prop}\label{prop: U and V low rank}
      Let $\aalpha \in \cA$ be a parameter. suppose that the semi-reduced order model \eqref{eq: RPP} is feasible, and that following rank conditions are satisfied :
      \begin{equation}
        \mathrm{rank}((\AA\widehat\WW)_x) = \mathrm{rank}(\widehat \UU^T (\AA \widehat\WW)_x),    
      \end{equation}
       and
       \begin{equation}
           \mathrm{rank}((\AA\widehat\WW)_y) = \mathrm{rank}(\widehat \VV^T (\AA \widehat\WW)_y).
       \end{equation}
    Then the half reduced- problem \eqref{eq: RPP} and the reduced problem \eqref{eq: RP} have the same optimal value.
\end{prop}
Before proving this Proposition, we need the following simple linear algebra Lemma.
\begin{lem}\label{lemma: A = ABC}
    Let $\mathbf{B} \in \RR^{D_1 \times D_2}$ and $\mathbf{E}\in \RR^{D_2 \times D_3}$ be matrices such that $\mathrm{Rank}(\textbf{B}) = \mathrm{Rank}(\textbf{BE}) $. Then, there exists a matrix $\textbf{F} \in \RR^{D_3 \times D_2}$ such that :
    \begin{equation}
        \textbf{B} = \textbf{BEF}.
    \end{equation}
\end{lem}
\begin{proof}
    In general, the inclusion $\mathrm{Im}(\textbf{BE}) \subset \mathrm{Im}(\textbf{B})$ holds. But since by hypothesis, the rank of the two matrices of consideration are equal. we have equality of the image spaces:
    \begin{equation}
        \mathrm{Im}(\textbf{BE}) = \mathrm{Im}(\textbf{B}).
    \end{equation}
    Thus, for every $1\leq d \leq D_2$, by hypothesis, $\textbf{B}\ee_d \in \mathrm{Im}(\textbf{B}) = \mathrm{Im}(\textbf{BE})$ (here $\{\ee_d\}_{1 \leq d\leq D_2}$ are the canonical basis' vector). Therefore, there exists a vector $\ff_d  \in \RR^{D_3}$ such that $\textbf{B}\ee_d =\textbf{BE}\ff_d$.
    Stacking every $\{\ee_d\}_{1\leq d \leq D_2}$ columnwise, we obtain 
    \begin{equation}
        \textbf{B} = \textbf{B}\textbf{Id}_{D_2} = \textbf{BEF},
    \end{equation}
    where $\textbf{F}$ is obtained by stacking columnwise the $\{\ff_i\}_{1 \leq i \leq m}$.
\end{proof}
We now ready to prove Proposition \ref{prop: U and V low rank}.
\begin{proof}
    We are going to show that the feasible sets of problems \eqref{eq: RPP} and \eqref{eq: RP} are the same. First, a feasible point of \eqref{eq: RPP} is feasible for \eqref{eq: RP}. So, we have to a feasible point \eqref{eq: RP} is also feasible for \eqref{eq: RPP}. For sake of clarity, we only focus on the upper part of the constraint equality, that is we prove that if $\pp(\aalpha)$ satisfies:
    \begin{equation}
        \UU^T (\AA \widehat \WW)_{x} \pp(\aalpha) = \UU^T \mmu(\aalpha), 
    \end{equation}
    then 
    \begin{equation}
    (\AA \widehat \WW)_{x} \pp(\aalpha) = \mmu(\aalpha). 
    \end{equation}
    By hypothesis, and since the rank is not affected by transposition, the rank equality $\mathrm{rank}((\AA \WW)^{T}_x) = \mathrm{rank}((\AA \WW)^{T}_x \UU)$ holds. From Lemma~\ref{lemma: A = ABC}, an by transposing again, we deduce that there exists a matrix $\textbf{F}$ such that $(\AA\widehat\WW)_x = \textbf{F}^T \UU^T (\AA\widehat\WW)_x$.
    So, let $\aalpha \in \cA$ be a parameter, and $\pp(\aalpha) \in \RR^R$ satisfying:
    \begin{equation} \label{proof eq: equality}
        \UU^T (\AA\widehat\WW)_x \pp(\aalpha)  = \UU^T \mmu(\aalpha).
    \end{equation}
    Since we supposed that the feasible set of \eqref{eq: RPP} is not empty, there exists a vector $\textbf{s}(\aalpha)$ such that 
    \begin{equation}
         \mmu(\aalpha) = (\AA \widehat\WW)_x \textbf{s}(\aalpha) = \textbf{F}^T \UU^T (\AA \widehat \WW)_x \textbf{s}(\aalpha) .
    \end{equation}
    And therefore, mutliplying \eqref{proof eq: equality} by $\textbf{F}^T$, we obtain
    \begin{equation}
      (\AA\widehat\WW)_x \pp(\aalpha) =   \textbf{F}^T  \UU^T (\AA\widehat\WW)_x \pp(\aalpha)  = \textbf{F}^T \UU^T  \mmu(\aalpha)  = \textbf{F}^T \UU^T (\AA \widehat\WW)_x \textbf{s}(\aalpha) = \mmu(\aalpha).
    \end{equation}
    An equivalent reasoning leads to the fact that $\pp(\aalpha)$ satisfies the lower part of the constraint equality.
\end{proof}
We  shortly notice that if we find $\widehat\UU$ and $\widehat\VV$ satisfying the rank assumptions, there is ne need to add constant vectors to them to ensure boundedness of the problem as we did in Section \ref{subsection: compactness}.
\begin{cor}
    Under the hypotheses of Property \ref{prop: U and V low rank}, the reduced-problem \eqref{eq: RP} is bounded.
\end{cor}
\begin{proof}
    The feasible set of \eqref{eq: RP} is equal to the feasible set of \eqref{eq: RPP}, which is bounded.
\end{proof}
\subsubsection{Gram-Schmidt algorithm to build the matrices \texorpdfstring{$\widehat{\UU}$ and $\widehat{\VV}$}{} }
In this part, we propose a way, namely a Gram-Schmidt algorithm, to build $\widehat\UU$ and $\widehat\VV$ that satisfy the rank condition of Proposition \ref{prop: U and V low rank}. Moreover, $\widehat\UU$ and $\widehat\VV$ will be of dimension $\RR^{K_x \times \cN_x}$ (resp. $\RR^{K_y \times \cN_y}$). As before, only focus on the construction of $\widehat \UU$, because the construction of $\widehat \VV$  is analogous. 
To make the following method work, we place ourselves in that case where the generating family $\{\mmu_{k}\}_{1 \leq k \leq K_x}$ is a family of independent vectors (therefore of maximal rank $K_x$).
We first detail the construction in Algorithm \ref{alg: Gram-Schmidt}, and then we explain why $\widehat \UU$ build this way satisfies the condition of Property \ref{prop: U and V low rank}.
\begin{algorithm}
\caption{Gram-Schmidt process}\label{alg: Gram-Schmidt}
\begin{algorithmic}[1]
\State \textbf{Inputs:} $\{\mmu_k\}_{1 \leq k \leq K_x}$
\State \textbf{Outputs:} $\widehat\UU := \{\uu_k\}_{1 \leq k \leq K_x}$
\State \textbf{Initialisation:}
\State $\uu_1 = \mmu_1$

\For{$k = 2 , \dots, K_x$} 
\State $ \displaystyle \uu_k = \mmu_k - \sum_{z=1}^{k-1} \dfrac{\uu_z^T \mmu_k}{\| \uu_z \|^2} \uu_z$
\EndFor
\end{algorithmic}
\end{algorithm}

Let us verify that the matrix $\widehat\UU$ built this way satisfies the expected rank condition. This derives from direct consequences of the Gram-Schmidt orthogonalisation process.
\begin{lem}\label{lemma: Gram-Schmidt}
    Let $\{\uu_k\}_{1\leq k  \leq K_x}$ built via Gram-Schmidt algorithm \ref{alg: Gram-Schmidt} applied to the family of independant vectors $\{ \mmu_k\}_{1 \leq k \leq K_x}$. Then
    \begin{enumerate}
        \item For every $1\leq k \neq k' \leq K_x$, $\uu_k \perp \uu_{k'}$
        \item For $1\leq k \leq K_x$, $\mathrm{Im}(\mmu_1, \dots, \mmu_k) = \mathrm{Im}(\uu_1, \dots, \uu_k)$
        \item For  $2\leq k \leq K_x$, $\uu_k \in \mathrm{Im}(\mmu_1, \dots ,\mmu_{k-1})^{\perp}$
        \item For $1\leq k \leq K_x$, $\uu_k^T \mmu_k \neq 0$.
    \end{enumerate}
\end{lem}
\begin{proof}
    We prove this lemma recursively: for k=1, $\uu_1 = \mmu_1$, thus $\mathrm{Im}(\uu_1) = \mathrm{Im}(\mmu_1)$ and $\uu_1^T\mmu_1 = \| \mmu_1\|^2 > 0$.
    
    Let us now suppose the lemma true up to rank $k-1 \geq 1$.
    Let $1\leq k' \leq k-1$. Using item 1 up to rank $k-1$ we have :
    \begin{equation}
        \uu_{k'}^T \uu_k = \uu_{k'}^T \mmu_k - \sum_{z=1}^{k-1} \dfrac{\uu_z^T \mmu_k}{\| \uu_z \|^2} \uu_{k'}^T\uu_z
        = \uu_{k'}^T \mmu_k -  \dfrac{\uu_{k'}^T \mmu_k}{\| \uu_{k'} \|^2} \uu_{k'}^T\uu_{k'}  = 0,
    \end{equation}
    which proves 1 up to rank $k$.
    
    Now, by item 2 at rank $k-1$, $\uu_k$ can be seen as a vector combination of $\mmu_1, \dots, \mmu_{k-1}$ and $ \mmu_k$. Thus $\uu_k \in \mathrm{Im} (\mmu_1, \dots, \mmu_k)$, and again by item 2 on rank $k-1$,  
    \begin{equation}
        \mathrm{Im}(\uu_1, \dots, \uu_{k-1}) = \mathrm{Im}(\mmu_1, \dots, \mmu_{k-1}), 
    \end{equation}
    thus, 
    \begin{equation}
     \mathrm{Im}(\uu_1, \dots, \uu_k) \subset \mathrm{Im}(\mmu_1, \dots, \mmu_k)  . 
    \end{equation}
    For the converse inclusion, we know that 
    \begin{equation}
        \mmu_k = \uu_k + \sum_{z=1}^{k-1} \dfrac{\uu_z^T \mmu_k}{\| \uu_z \|^2} \uu_z, 
    \end{equation} 
    thus $\mmu_k \in \mathrm{Im}(\uu_1, \dots \uu_{k-1}, \uu_k)$.
    Again by item 2 applied on $k-1$ we obtain
    \begin{equation}
    \mathrm{Im}(\mmu_1, \dots, \mmu_k) \subset \mathrm{Im}(\uu_1, \dots, \uu_k).    
    \end{equation}
    We deduce equality from the two inclusions, and item 2 is true at rank $k$. 
    
    Item 3 is a consequence of items 1 and 2 :  since $\uu_k \perp \uu_{k'}$ for $1\leq k' \leq k-1$, $\uu_k \in \mathrm{Im}(\uu_1, \dots, \uu_{k-1})^\perp$ By item 2, $\mathrm{Im}(\uu_1, \dots, \uu_{k-1})^\perp = \mathrm{Im}(\mmu_1, \dots, \mmu_{k-1})^\perp$, thus $\uu_k \in \mathrm{Im}(\mmu_1, \dots, \mmu_{k-1})^\perp$.

    Finally, item 4 is a consequence of items 1 and 2 conjugated to the fact that, by hypothesis, $\mmu_k$ is independent of $\mmu_1, \dots \mmu_{k-1}$. By independence of the family $(\mmu_1, \dots,\mmu_k)$, and by item 2, $\uu_k \neq 0$. Indeed, otherwise we would have : 
    \begin{equation}
        0 = \mmu_k - \sum_{z=1}^{k-1} \dfrac{\uu_z^T \mmu_k}{\| \uu_z \|^2} \uu_z,
    \end{equation}
    thus $\mmu_k \in \mathrm{Im}(\uu_1, \dots, \uu_{k-1}) = \mathrm{Im}(\mmu_1, \dots, \mmu_{k-1})$, and that would contradict the independency hypothesis. Let us now write the scalar product and use item $1$:
    \begin{equation}
        \uu_k^T \mmu_k = \uu_k^T \left(\uu_k + \sum_{z = 1}^{k-1} \dfrac{\uu_{z}^T \mmu_k}{\| \uu_z \|^2} \uu_k^T\uu_z\right) = \| \uu_k\|^2> 0, 
    \end{equation}
    which gives item 4 at rank $k$.
    
\end{proof}
\begin{cor}\label{corollary: GS rk K0}
Let $\widehat\UU$ be the matrix build with Gram-Schmidt algorithm \ref{alg: Gram-Schmidt} on the family $\{\mmu(\aalpha) \}_{\aalpha \in \cA}$. Then $\widehat\UU^T \begin{pmatrix} \mmu_0 | \dots |\mmu_{K_x}  \end{pmatrix}$ is of rank $K_x$.
\end{cor}
\begin{proof}
   Using the orthogonality conclusion 3 of the lemma \ref{lemma: Gram-Schmidt}, we have :
    \begin{equation}
        \widehat\UU^T \begin{pmatrix} \mmu_0 | \dots | \mmu_{K_x} \end{pmatrix} = \begin{pmatrix} \uu_1^T \mmu_1 &\dots& \uu_1^T \mmu_{K_x} \\
        \vdots && \vdots  \\
        \uu_{K_x}^T \mmu_1  &\dots &\uu_{K_x}^T \mmu_{K_x}
        \end{pmatrix} =
        \begin{pmatrix}
             \uu_1^T \mmu_1 &\dots&& \uu_1^T \mmu_{K_x} \\
             0 & \uu_2^T \mmu_2 & \dots &\uu_2^T\mmu_{K_x}& \vdots \\
        \vdots & \ddots & \ddots &\\
        0 &\dots & 0&\uu_{K_x}^T \mmu_{K_x} 
        \end{pmatrix}.
    \end{equation}
    By conclusion 4 in the lemma \ref{lemma: Gram-Schmidt}, this upper triangular matrix has non zero diagonal and is thus of maximal rank $K_x$.
\end{proof}
\begin{cor} \label{corollary: GS satisfies rank condition}
    Let $\widehat\UU$ be the matrix build with Gram-Schmidt algorithm \ref{alg: Gram-Schmidt} on the family $\{\mmu(\aalpha) \}_{\aalpha \in \cA}$. Then $\mathrm{rank}(\widehat\UU^T (\AA \widehat \WW)_x)= \mathrm{rank}((\AA \widehat \WW)_x) $.
\end{cor}
To prove Corollary \ref{corollary: GS satisfies rank condition} we use the following lemma of linear algebra.
\begin{lem}\label{lemma: rank condition stable by composition}
    Let $D_1,D_2,D_3$ be positive integers and $\textbf{E} \in \RR^{D_1\times D_2}$ and $\textbf{F} \in \RR^{D_2 \times D_3}$ be matrices. Suppose that $\mathrm{rank}(\textbf{EF}) = \mathrm{rank}(\textbf{F})$. Then, for any positive integer $Q$, for any matrix $\textbf{M} \in \RR^{D_3 \times Q}$, $\mathrm{rank}(\textbf{EFM}) = \mathrm{rank}(\textbf{FM})$.
\end{lem}
\begin{proof}
    In general, $\mathrm{ker}(\textbf{F}) \subset \mathrm{ker}(\textbf{EF})$. But, since $D_2= \mathrm{rank}(\textbf{F}) + \dim(\ker(\textbf{F})) = \mathrm{rank}(\textbf{EF}) + \dim(\ker(\textbf{EF}))$, the rank equality implies $\mathrm{ker}(\textbf{F}) = \mathrm{ker}(\textbf{EF})$. Thus 
    \begin{equation}
        \mathrm{ker}(\textbf{FM}) = \left\lbrace \ff \in \RR^Q, \textbf{M}\ff \in \mathrm{ker}(\textbf{F)} \right\rbrace 
        =   \left\lbrace \ff \in \RR^Q, \textbf{M}\ff \in \mathrm{ker}(\textbf{EF)} \right\rbrace
        = \mathrm{ker}(\textbf{EFM}).
    \end{equation}
    Therefore, $\mathrm{rank(\textbf{EFM})} = \mathrm{rank(\textbf{FM})} $.
\end{proof}
We now prove Corollary \ref{corollary: GS satisfies rank condition}.
\begin{proof}
    Since every column of $\mathrm{rank}((\AA \widehat \WW)_x)$ is a convex combination of the family of vectors $\{ \mmu_{k} \}_{1 \leq k \leq K_x}$, there exists a matrix $\textbf{M}$ such that :
    \begin{equation}
        \begin{pmatrix} \mmu_0 | \dots | \mmu_{K_x} \end{pmatrix} \textbf{M} = \begin{pmatrix} \mmu(\aalpha_1) | \dots | \mmu_{K_x}\end{pmatrix} = (\AA \WW)_x.
    \end{equation}
    From corollary \ref{corollary: GS rk K0}, the hypothesis of lemma \ref{lemma: rank condition stable by composition} hold with \ref{lemma: rank condition stable by composition} with $\textbf{E} = \widehat{\UU}^T$ and $\textbf{F} =  \begin{pmatrix} \mmu_0 | \dots | \mmu_{K_x} \end{pmatrix}$. Thus, taking $\textbf{M}$ as defined above, we obtain :
    \begin{equation}
       \mathrm{rank}\big(\widehat \UU^T (\AA \WW)^x \big) = \mathrm{rank}\big( \widehat \UU^T \begin{pmatrix} \mmu_0 | \dots | \mmu_{K_x} \end{pmatrix}\textbf{M} \big) =  \mathrm{rank}\big(\begin{pmatrix} \mmu_0 | \dots | \mmu_{K_x} \end{pmatrix} \textbf{M} \big)  = \mathrm{rank}\big( \AA \WW)_x \big).
    \end{equation}
\end{proof}

\section{A posteriori error estimations}\label{section : A posteriori error approximator}
In the previous section, we focused on building a reduced-order model \eqref{eq: RP} from the high-fidelity model \eqref{eq: primal}, and we found conditions ensuring that the model built had solutions. Now, let $\aalpha \in \cA$ be a parameter,  suppose that we have computed numerically the optimal value $I_R(\aalpha)$ of the reduced-order model at $\aalpha$. We would like to estimate how far this value is from the optimal cost $I(\aalpha)$ of the high-fidelity model. To do so, we search for quantities that bound the value $|I(\aalpha) - I_R(\aalpha)|$, and that can be computed without having to solve the high-fidelity model at $\aalpha$. We call such a quantity an a posteriori error estimation.
We present two error estimations. The first one is built using C-transforms of a pair of potentials solution of the reduced-order model \eqref{eq: RD}. The second one is based on the continuity of the model with respect to the parameter $\aalpha$.
%
\subsection{Error estimation with \texorpdfstring{$c$}{}-transforms}\label{subsection : error approximation $c$-transform}
For this estimation, we suppose that the reduced-order model \eqref{eq: RP} and the half reduced-order model on primal \eqref{eq: RPP} share the same optimal value (this is satisfied when the conditions of Property \ref{prop: U and V low rank} hold). Since the half reduced-order model \eqref{eq: RPP} corresponds to the minimization problem described by the high-fidelity model \eqref{eq: primal} to which we add the constraint to belong to the subcone $\cW_+$, its optimal value is higher than the optimal value of the high-fidelity problem. We thus deduce directly the following inequality for every parameter $\aalpha \in \cA$ : 
\begin{equation}
    I_R(\aalpha) - I(\aalpha) \geq 0.
\end{equation}
 We thus focus on bounding the quantity $ I_R(\aalpha) - I(\aalpha) $ from above.
Let us first recall, that, by strong duality, the optimal values of the primal and dual of the high-fidelity model (resp. of the reduced-order model) are equal : $J(\aalpha) = I(\aalpha)$ and $J_R(\aalpha) = I_R(\aalpha)$. %
Now, let $(\aa^*(\aalpha), \bb^*(\aalpha))$ be a solution of \eqref{eq: RD}  and $\widehat{\pphi}_*(\aalpha)$ and $\widehat{\ppsi}_*(\aalpha)$ the associated Kantorovitch potentials:
\begin{equation} \label{eq: bound below phi psi definition}
    \begin{array}{ll}
         \displaystyle  \widehat{\pphi}_*(\aalpha) := \sum_{n =1}^N \aa^*_n(\aalpha) \uu_n \\
         \displaystyle  \widehat{\ppsi}_*(\aalpha) := \sum_{m =1}^M \bb^*_m(\aalpha) \vv_m .
    \end{array}
\end{equation}
Since this pair is optimal, the followig equality holds :
\begin{equation} \label{eq: proof reduced cost}
    I_R(\aalpha) = \left\langle \widehat{\pphi}_*(\aalpha), \mmu(\aalpha) \right\rangle + \left\langle \widehat{\ppsi}_*(\aalpha), \nnu(\aalpha) \right\rangle .
\end{equation}
We now use the $c$-transform and $\bar{\mathrm{c}}$-transform defined in \ref{def: $c$-transform} to create admissible points of the dual high-fidelity model \eqref{eq: dual}.
Indeed, the pairs $(\widehat{\pphi}_*(\aalpha), \widehat{\pphi}^{c}_*(\aalpha))$ and $(\widehat{\ppsi}^{\bar{c}}_*(\aalpha), \widehat{\ppsi}_*(\aalpha))$ are admissible for \eqref{eq: dual} (see e.g. \cite{santambrogio2015optimal} for a proof).
Let us focus on the pair $(\widehat{\pphi}_*(\aalpha), \widehat{\pphi}^{c}_*(\aalpha))$. Since it is admissible, we have the following inequality : 
\begin{equation}\label{eq: proof inequality opt value}
    \begin{array}{ll}
         I(\aalpha) &\geq \left\langle \widehat{\pphi}_*(\aalpha), \mmu(\aalpha) \right\rangle +
                        \left\langle \widehat{\pphi}^{c}_*(\aalpha), \nnu(\aalpha) \right\rangle.  \\
    \end{array}
\end{equation}

We now substract \eqref{eq: proof reduced cost} to \eqref{eq: proof inequality opt value}, and we obtain the following expression :
\begin{equation}
    \begin{array}{ll}
     I_R(\aalpha)   - I(\aalpha)  &\leq \left\langle  \widehat{\ppsi}_*(\aalpha)-  \widehat{\pphi}^{c}_*(\aalpha), \nnu(\aalpha) \right\rangle.   \\
    \end{array}
\end{equation}
By performing the same computation with $(\widehat{\ppsi}^{\bar{c}}_*(\aalpha), \widehat{\ppsi}_*(\aalpha))$, we obtain :
\begin{equation}
    \begin{array}{ll}
         I_R(\aalpha) - I(\aalpha)  &\leq \left\langle \widehat{\pphi}_*(\aalpha)- \widehat{\ppsi}^{\bar{c}}_*(\aalpha), \mmu(\aalpha) \right\rangle.   \\
    \end{array}
\end{equation}
We thus obtain the following a posteriori error estimation that bounds from above the difference $ I_R(\aalpha)  - I(\aalpha)$ :
\begin{equation}
     I_R(\aalpha)  - I(\aalpha) \leq \min \left(\left\langle \widehat{\pphi}^{c}_*(\aalpha) - \widehat{\ppsi}_*(\aalpha), \nnu(\aalpha) \right\rangle, \left\langle \widehat{\ppsi}^{\bar{c}}_*(\aalpha) - \widehat{\pphi}_*(\aalpha), \mmu(\aalpha) \right\rangle  \right).   
\end{equation}

Since this estimation uses $c$-transforms and $\bar{c}-$transforms, in general, it depends non linearly on the parameter $\aalpha$. We show in what follow that the exact computation of this estimation is of order $\cN_x\cN_y$. We then propose a way to approach this quantity by a variant of the Empirical Interpolation Method (EIM).

\paragraph{Evaluation of the computational }
We focus here on the computational cost of the term $\left\langle \widehat{\pphi}^{c}_*(\aalpha) - \widehat{\ppsi}_*(\aalpha), \nnu(\aalpha) \right\rangle$. We can deduce by analogy the number of operations needed to evaluate $\left\langle \widehat{\ppsi}^{\bar{c}}_*(\aalpha) - \widehat{\pphi}_*(\aalpha), \mmu(\aalpha) \right\rangle$.

We begin by replacing $\widehat\ppsi_*(\aalpha)$ with its expression in \eqref{eq: bound below phi psi definition}, and we express $\nnu(\aalpha)$ as a convex combination of $\{\nnu^{l} \}_{l = 1}^{K_y}$. We obtain a weighted sum of $M K_y$ terms:
\begin{equation}
    \left\langle\widehat{\ppsi}_*(\aalpha), \nnu(\aalpha) \right\rangle = 
    \sum_{m=1}^M \sum_{l = 1}^{K_y} \alpha^y_{l} b_m    \left\langle \vv^*_m, \nnu^{l} \right\rangle.
\end{equation}
The expression $\left\langle \widehat{\pphi}^{c}_*(\aalpha), \nnu(\aalpha) \right\rangle$ can be decomposed as follows:
\begin{equation}
    \left\langle \widehat{\pphi}^{c}_*(\aalpha) , \nnu(\aalpha) \right\rangle = 
   \sum_{l = 1}^{K_y} \alpha^y_{l}    \left\langle \widehat\pphi^{c}_*(\aalpha), \nnu^{l} \right\rangle.
\end{equation}
In general, the $c-$transform operation does not depend linearly on the parameter $\aalpha$, and its exact evaluation requires $\cN_y \cN_x$ operations.


\paragraph{EIM for fast evaluation of the error estimation}\label{subsection: EIM}
For a given $\aalpha \in \cA$, the exact evaluations of $\widehat\pphi^c_*(\aalpha)$ and of $\left\langle \widehat\pphi^c_*(\aalpha), \nnu(\aalpha)\right\rangle$ both require around $\cN_x\cN_y$ operations. In this section we implement a variant of the Empirical Interpolation Method (EIM) that enables to estimate $\widehat\pphi^c(\aalpha)$ at a reduced computational cost. (see e.g. \cite{maday2007general} or \cite{barrault2004empirical} for a general description of the EIM). 
More specifically, here our aim is to approximate $\widehat\pphi_*^c(\aalpha)$ as a linear combination of $M^{\text{eim}}$ independent vectors $\{\qq^{\ell}\})_{1 \leq \ell \leq M^{\text{eim}}}$:
\begin{equation}
    \widehat\pphi^c_*(\aalpha) \approx \sum_{\ell = 1}^{M^\eim} \bbeta_\ell(\aalpha) \qq^\ell.
\end{equation}
The coefficients $\{\bbeta^{\ell}(\aalpha)\})_{1 \leq \ell \leq M^{\text{eim}}}$ are found by solving a linear system of size $M^{\text{eim}^2}$, where the right hand side vector is formed of $M^{\text{eim}}$ fast evaluations of the $c$-transform of $\widehat \pphi_*(\aalpha)$:
\begin{equation}
    \sum_{\ell = 1}^{M^\eim} \qq^\ell_j \beta_{\ell}(\aalpha) = \widehat\pphi^{\text{Fast-C}}_{*j}(\aalpha), \; \text{for } j \in J^\eim,
\end{equation}
where 
\begin{equation}
    \widehat\pphi^{\text{Fast-C}}_{*j} := \arg\min_{i \in I^\eim } \CC_{i,j} - \widehat\pphi_{*i}(\aalpha).
\end{equation}
How to find $\{\qq^{\ell}\})_{1 \leq \ell \leq M^{\text{eim}}}$, the evaluation points and how to build a fast evaluation are determined by the offline phase of the EIM described in Algorithm \ref{alg: EIM}.
\begin{algorithm}
\caption{EIM (offline phase)}\label{alg: EIM}
\begin{algorithmic}[1]
\State \textbf{Inputs:} $\mathcal{G}:=\{\widehat\pphi (\aalpha)\}_{\aalpha \in \cA_{train}}$, $\mathcal{G^C}:=\{\widehat\pphi^C (\aalpha)\}_{\aalpha \in \cA_{train}}$  $M^{\text{eim}}, M^{'\text{eim}}$
\State \textbf{Outputs:} $I^\eim \subset \{1, \dots, \cN_x\}, J^\eim \subset \{1, \dots, \cN_y\}$, $\{\qq^1, \dots, \qq^{M^\eim} \} \in \RR^{\cN_y \times M^{\eim}}$
\State \textbf{Initialisation:}
\State $\Gg^1 = \arg\max_{\widehat\pphi^C \in \cG^C} |\widehat\pphi^C|$
\State $j_1 = \arg\max_{1\leq j \leq \cN_y} |\Gg^1_{j}|$
\State $\qq^1 = \dfrac{\Gg^1}{\Gg ^1_{j_1}}$
\State $\cI_1[\widehat\pphi] = \widehat\pphi_{j_1} \qq^1$
\State $J^\eim \xleftarrow{} \{j_1\}$

\For{$\meim = 2 , \dots, M^{\text{eim}}$} 
\State $\Gg^{\meim} = \arg\max_{\widehat\pphi^C \in \cG^C} \|\widehat\pphi^C - \cI_{\meim -1}[\widehat\pphi^C]\|$
\State $j_{\meim} = \arg\max_{1\leq j \leq \cN_y} |\Gg^{\meim}_{j}|$
\State $\qq^{\meim} = \dfrac{\Gg^{\meim} -  \cI_{\meim -1}[\Gg^{\meim}]}{(\Gg^{\meim} - \cI_{\meim -1}[\Gg^{\meim}])_{j_{\meim}}}$
\State $\cI_{\meim}[\widehat\pphi^C] = \sum_{\ell=  1}^{\meim} \beta_{\ell} \qq^{\ell}, $
\State where $\beta$ solves $\sum_{\ell = 1}^{\meim} \beta_{\ell} \qq^{\ell}_{j_l} = \widehat{\pphi}^C_{j_l} $ for $1 \leq l \leq \meim$
\State $J^\eim \leftarrow J \cup \{j_{\meim}\}$
\EndFor

\For{$1\leq \meim_1 \leq M^{\text{eim}}$ }
    \For{$1\leq \meim_2 \leq M^{'\text{eim}}$ }
    \State $i_{\meim_1, \meim_2} = \arg\min_{1\leq i\leq \cN_x} \CC_{i,j_{\meim_2}} - \Gg^{\meim_1}_{j_{\meim_2}}$
    \State $I^\eim \leftarrow \{i_{\meim_1,\meim_2}\} $
    \EndFor
\EndFor
\end{algorithmic}
\end{algorithm}

In the end, we have the following approximation:

\begin{equation}\label{eq: EIM cost}
    \left\langle \widehat\pphi^c(\aalpha), \nnu(\aalpha)\right\rangle \approx
    \sum_{l = 1}^{K_y} \sum_{\ell =1}^{M^{\text{eim}}} \bbeta_{\ell} \aalpha^y_{l} \left\langle \qq^\ell, \nnu_{l} \right\rangle.
\end{equation}

In term of computational costs, the determination of $\{\bbeta_{\ell}\}_{1 \leq \ell \leq M^{\text{eim}}}$ requires of order $M^{\text{eim}}M^{'\text{eim}}$ for finding the right-hand side, then of order $M^{\text{eim}^2}$ to solve the system (it has a triangular structure). The final scalar product requires around $K_yM^{\text{eim}}$ operations.

\subsection{Error estimation from continuous dependency in \texorpdfstring{$\aalpha$}{}} \label{subsection: erorr estimation continuity}
In this part, we use the continuity Property \ref{prop: contiuity with respect to alpha} to build an upper bound of the quantity $|I(\aalpha) - I_R'(\aalpha)|$.
Let $\cA_{\text{train}} \subset \cA$ be a training set, and suppose that we know the optimal cost of the high-fidelity problem $I({\aalpha'})$ for every parameter $\aalpha' \in \cA_{\text{train}}$ of this training set.
Let $\aalpha \in \cA$ be a parameter, then, for very $\aalpha'\in \cA_{\text{train}}$, we have :
\begin{equation}
    |I(\aalpha) - I_R(\aalpha)| \leq |I(\aalpha) - I(\aalpha')| + |I(\aalpha') - I_R(\aalpha)|.
\end{equation}
Using Property \ref{prop: contiuity with respect to alpha}, we obtain
\begin{equation}
    |I(\aalpha) - I_R(\aalpha)| \leq |I(\aalpha') - I_R(\aalpha)| + \|\CC\|_\infty  (2 \max(K_x, K_y) + 3 \min(K_x, K_y)) \|\aalpha - \aalpha'\|_\infty.
\end{equation}
We can then control this difference by choosing a suitable $\aalpha' \in \cA_{\text{train}}$, for example $\displaystyle \aalpha' \in \arg\min_{\aalpha' \in \cA_{\text{train}}} \big( |I(\aalpha') - I_R(\aalpha)| + \|\CC\|_\infty  (2 \max(K_x, K_y) + 3 \min(K_x, K_y)) \|\aalpha - \aalpha'\|_\infty \big)$ or  $\displaystyle \aalpha' \in \arg\min_{\aalpha' \in \cA_{\text{train}}}  \|\aalpha - \aalpha'\|_\infty$.
The complexity of this estimation depends on the cardinal $\#\cA_{\text{train}}$ of $\cA_{\text{train}}$ and of the way to choose $\aalpha'$. For the proposed $\aalpha'$ it will be of order $\#\cA_{\text{train}}$.

\section{Numerical results}\label{sec:num}
In this section, we first use a toy example to give some numerical results about our method. More precisely, we analyse on this example simultaneously the time gain and precision obtained with our method in comparison to solving the high-fidelity problem with the linear programming solver \texttt{linprog} of the Python library \texttt{scipy.optimize}. We compare these results to the time gain and precision of the classical Sinkhorn algorithm. We then test the behaviors of the two a posteriori error estimations we built.
 Finally, we apply our reduction method to a problem of 3D optimal transport applied to a transfer of color palette. 
Every numerical experiment can be found on the two notebooks \url{https://colab.research.google.com/drive/1OG9mqaIrASvS_ucC-V9sngRppLmtoOnr?usp=sharing} and \url{https://colab.research.google.com/drive/1m7X29BdsbFLV5iXbk7oe5_OKY_lsga76?usp=sharing}.
\begin{rmk}
   In each presented reduced-order model, the dual reduced basis is built using the method presented in Section \ref{subsection : small U and V}. Thus, when the size of a reduced-basis is mentioned, it refers to the dimension of the primal reduced-cone.
\end{rmk}
%

\subsection{Analysis of the numerical results of the reduced-order model in dimension 1}
In the sequel, we use a simple example of parametrized optimal transport problem to analyze numerically the performances of the reduced-order model \eqref{eq: RP} in term of time gain and precision with respect to the high-fidelity model \eqref{eq: primal}. We begin by describing the high-fidelity model built. We then briefly display at a fixed parameter solutions of the reduced-order model for different sizes of reduced bases, and compare it to a reference solution of the high-fidelity model.

\paragraph{The setting}
The family of optimal transport problems we aim at studying is a 1-dimensional optimal transport problem on the domains $X = \{x_i=  -1 + \dfrac{2i}{\cN_x}\}_{1 \leq i \leq \cN_x}$ and $Y= \{y_j = -1 + \dfrac{2j}{\cN_y}\}_{1 \leq j \leq \cN_y}$. We chose the quadratic cost $\CC_{ij} = \|y_j - x_i \|^2$, for $1\leq i \leq \cN_x$, $1\leq j \leq \cN_y$. We chose the family of marginals to be convex combinations of 2 normalized gaussians given by:
\begin{equation}
    \mmu_0^i = \frac{e^{- \dfrac{(x_i - m^{\mu}_0)^2}{2 \sigma^{\mu 2}_0}}}{\sum_{i^" = 0}^{\cN_x} e^{-\dfrac{(x_{i^"} - m^{\mu}_0)^2}{2 \sigma_0^{\mu 2}}}}, 
     \mmu_1^i = \dfrac{e^{- \dfrac{(x_i - m^{\mu}_1)^2}{2 \sigma^{\mu 2}_1}}}{\sum_{i^" = 0}^{\cN_x} e^{- \dfrac{(x_{i^"} - m^{\mu}_1)^2}{2 \sigma^{\mu2}_1}}} \text{ for } 1 \leq i \leq \cN_x, 
,
\end{equation}
and
\begin{equation}
        \nnu_0^j = \dfrac{e^{- \dfrac{(y_j - m^{\nu}_0)^2}{2 \sigma^{\nu 2}_0})}}{\sum_{j^" = 0}^{\cN_y} e^{(\dfrac{(y_{j^"} - m^{\nu}_0)^2}{2 \sigma_0^{\nu 2}})}}, 
     \nnu_1^j = \dfrac{e^{(- \dfrac{(y_j - m^{\nu}_1)^2}{2 \sigma^{\nu 2}_1})}}{\sum_{j^" = 0}^{\cN_y} e^{- \dfrac{(y_{j^"} - m^{\nu}_1)^2}{2 \sigma^{\nu2}_1})}} \text{ for } 1 \leq j \leq \cN_y,
\end{equation}
where $m^{\mu}_0 = -\frac{1}{2}$, $m^{\mu}_1 = \frac{1}{2}$, $m^{\nu}_0 = -\frac{1}{2}$, $m^{\nu}_0 = \frac{1}{2}$, and $\sigma_0^{\mu} = \sigma_1^{\mu} = \sigma_0^{\nu} = \sigma_1^{\nu} = \frac{1}{2} $ . Figure \ref{fig: margs} gives an example of a parametrized measure obtained in this configuration. We set the discretizations parameters at $\cN_x = \cN_y = 100$. 
\begin{figure}[htbp] 
  \centering
  \includegraphics[width=0.9\linewidth]{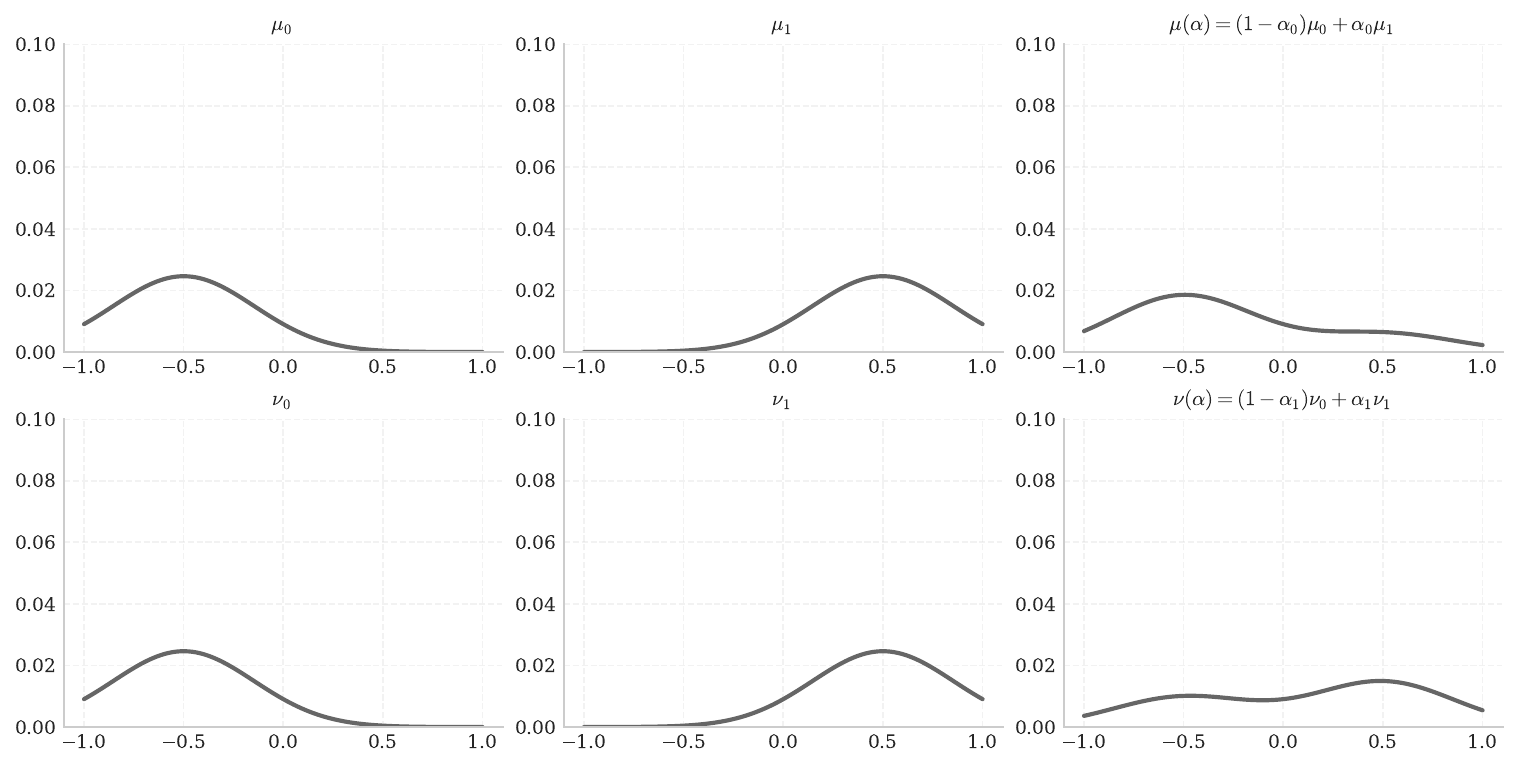}
\caption{Visualization of the parametrized marginals}
\label{fig: margs}
\end{figure}

\paragraph{Snapshot grid \texorpdfstring{$\cA_{\text{train}}^P$}{}}
We built the training set $\cA_{\text{train}}^P$ as a regular grid of $[0,1]^2$. We then made the precision vary from $4 = 2^2$ to $400 = 20^2$ snapshots.

\paragraph{Shape of solutions for the ROM}
On Figure \ref{fig: solutions of ROM} we show the general shape of the solutions of the reduced-order model (ROM) obtained with different sizes of reduced bases (RB), and we compare it to the solution of the high-fidelity model. By definition, solutions of the ROM are  positive (even convex) combinations of solutions of the high-fidelity model, and that is why in practice they have a sort of filament shape (and not the shape of a map anymore). The bigger the reduced bases are, the more the reduced solutions get a shape that looks similar to a map.
\begin{figure}[htbp] 
  \centering
  \includegraphics[width=0.99\linewidth]{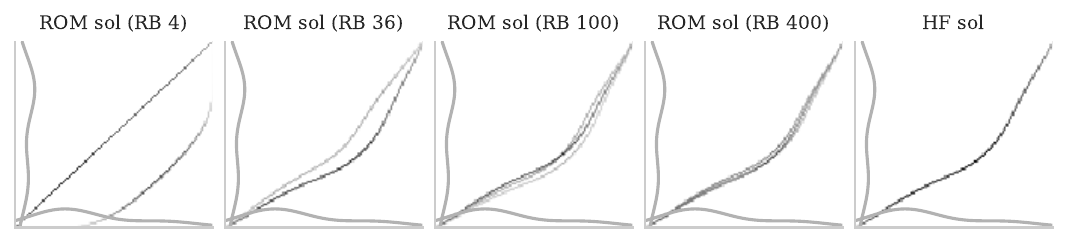}
\caption{Left : solutions of the reduced order-model (ROM) for different sizes of reduced bases (RB). Right : solution of the high-fidelity model (HF) of size $(10000\times 200)$.}
\label{fig: solutions of ROM}
\end{figure}

\paragraph{Evolution of the error with respect with the number of snapshots taken}
Figure \ref{fig: error} describes the evolution of mean error between the optimal value of the High-fidelity problem \eqref{eq: primal} and the reduced-order model \eqref{eq: RP} with respect to the number of snapshots taken to build the reduced basis. As expected,  the error decreases, as the snapshot grid gets finer.
\begin{figure}[htbp] 
  \centering
  \includegraphics[width=0.80\linewidth]{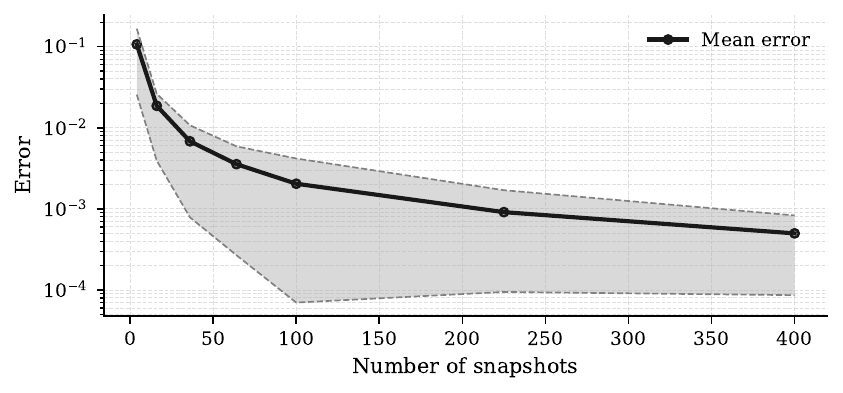}
\caption{Evolution of the mean error over a test set of 50 random parameters with respect to the number of snapshots taken. The grey area is delimited by the minimal and maximal error obtained on this test set.}
\label{fig: error}
\end{figure}

\paragraph{Comparison of performances of the reduced-order model with those of Sinkhorn}
We show on Figure \ref{fig: performances} the time gain and accuracy evolution of the reduced-order model for different sizes of reduced bases. We observe that, the bigger the reduced-basis gets, the less the error between the reduced-order model and the high-fidelity model is, but also the less the time gain is.
We also plot the results obtained with a Sinkhorn method, for different values of entropy $\varepsilon$. It appears that for the best precision points of both methods, the time gain of the reduced-order model is greater.
\begin{figure}[htbp] 
  \centering
  \includegraphics[width=0.92\linewidth]{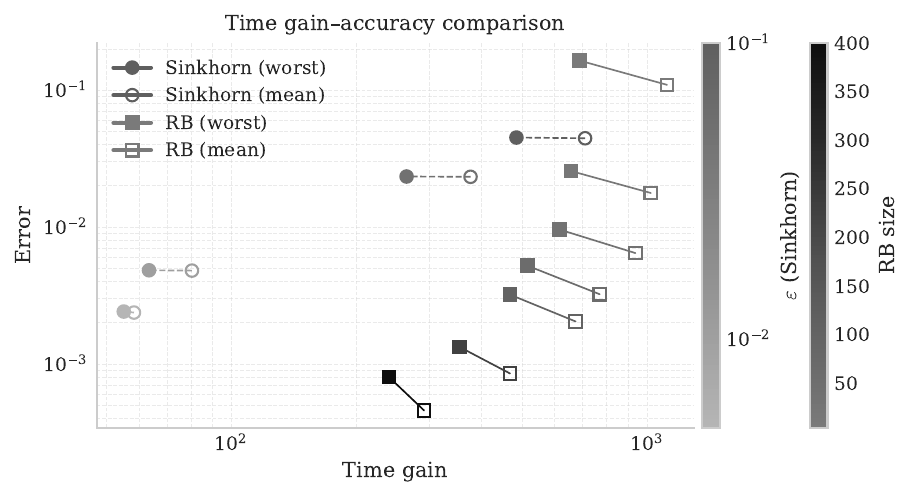}
\caption{Comparison of precision and time gain between the reduced-order model and the Sinkhorn method. We tested the method on a test set of $50$ random parameters, and plot in each case worst and mean performance.}
\label{fig: performances}
\end{figure}
%

\paragraph{A posteriori error estimations with \texorpdfstring{$c$}{}-transforms}
In the sequel, we test the first a posteriori error estimation build in Section \ref{section : A posteriori error approximator}.
We begin by comparing the exact error estimation to its fast evaluation.  We then compare the real error of the reduced-order model with the estimated one. It appears that the a posteriori error estimation may be far from the real error. We thus suggest a way to estimate the gap between the real error and the estimation in order to obtain a more accurate error estimation.

\paragraph{Comparison between the exact a posteriori error estimation and its fast evaluation }
We test here the accuracy of the fast estimation presented in Section \ref{subsection: EIM} that uses an EIM interpolation. Figure \ref{fig: test EIM evaluation} shows the evolution of the gap $|\text{fast error estimation} - \text{exact error estimation}|$ with respect to the size $M^{\eim}$ of the EIM basis. More precisely, it shows the mean, minimum and maximum gaps computed on a test set of $50$ random parameters.  It appears that this difference goes rapidly beyond the threshold $10^{-1}$ but then tends to diminish slowly.

 \begin{figure}[htbp] 
  \centering
  \includegraphics[width=0.8\linewidth]{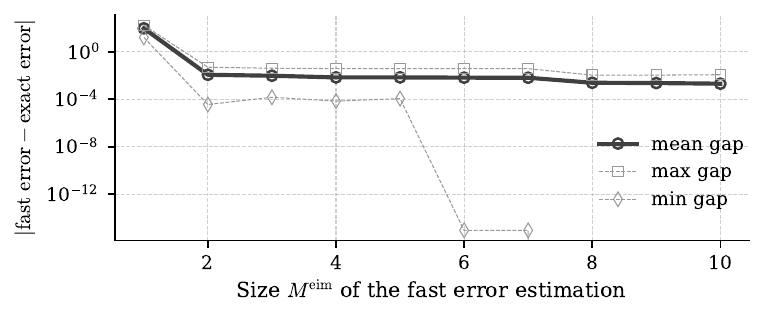}
  \caption{Gap between the fast evaluation and the exact evaluation of the a posteriori error estimation in function of the size of the EIM estimation taken. The curves are built as a mean/min/max of this gap over a test set of 50 random parameters.  }
  \label{fig: test EIM evaluation}
\end{figure}
\paragraph{Precision of the error estimation}
Figure \ref{fig: EIM estimator, fast U and V} compares on a test set of $50$ random parameters the mean exact error and the mean error obtained with the a posteriori error estimation using $c$-transforms.
 In the following, we compare the exact a posteriori error estimation and its fast approximation to the real error between the cost of the High-fidelity model and the reduced-order model. On figure \ref{fig: EIM estimator, fast U and V}, we observe that both fast - and exact error estimations behave poorly when the matrices $\UU$ and $\VV$ are built via the method proposed in \ref{subsection : small U and V}.
 \begin{figure}[htbp] 
  \centering
  \includegraphics[width=0.6\linewidth]{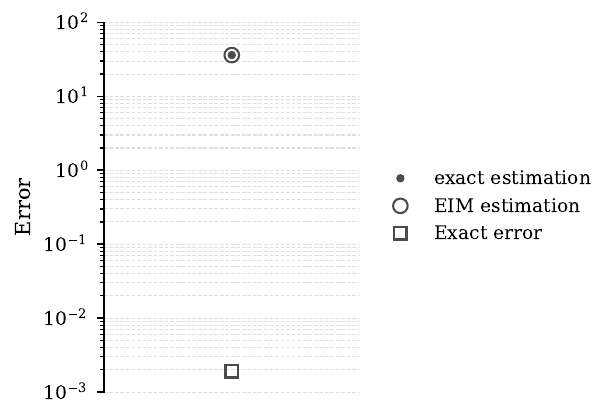}
  \caption{Comparison of the exact error and the a posteriori error estimation. Here we represent a mean of errors computed on a test set of 50 random parameters. }
  \label{fig: EIM estimator, fast U and V}
\end{figure}

\paragraph{Corrected error estimation}\label{paragraph:  error correction}
In order to get a better estimation, we estimate using a training set $\cA_{\text{est}}$ three constants $C_{\max}$, $C_{\min}$ and $C_{\text{mean}}$ that correspond to the maximum (respectively the minimum and the mean) of the ratio between the real error and the error approximation evaluated on every parameter of $\cA_{\text{test}}$. We then multiply the a posteriori error estimation by these constants. Using this method, we obtain a more precise estimation, as shown on figure \ref{fig: error correction}.
\begin{figure}[htbp] 
  \centering
  \includegraphics[width=0.68\linewidth]{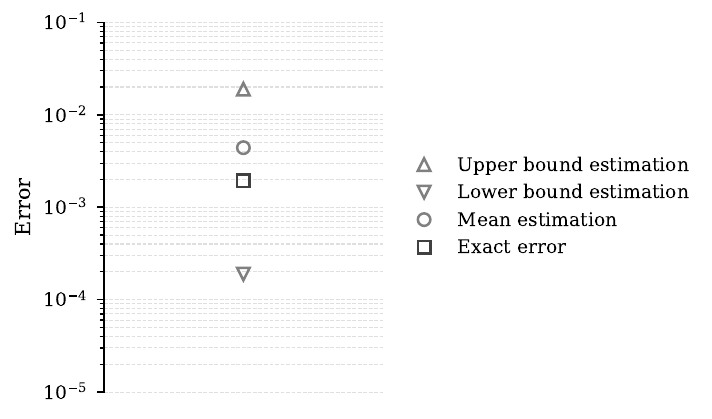}
  \caption{Error correction with a multiplicative constant. The values correspond to a mean over 50 random parameters. The training set used to approximate the constant is of size 100.}
  \label{fig: error correction}
\end{figure}
%

\paragraph{A posteriori error estimations using continuity}
Figure \ref{fig: estimation continuity} shows the a posteriori error estimation obtained using the bound computed in Section. \ref{subsection: erorr estimation continuity}. The bigger the training set, the finer the estimation, but we are still loosing several order of magnitude in comparison to the real error.
\begin{figure}[htbp] 
  \centering
  \includegraphics[width=0.8\linewidth]{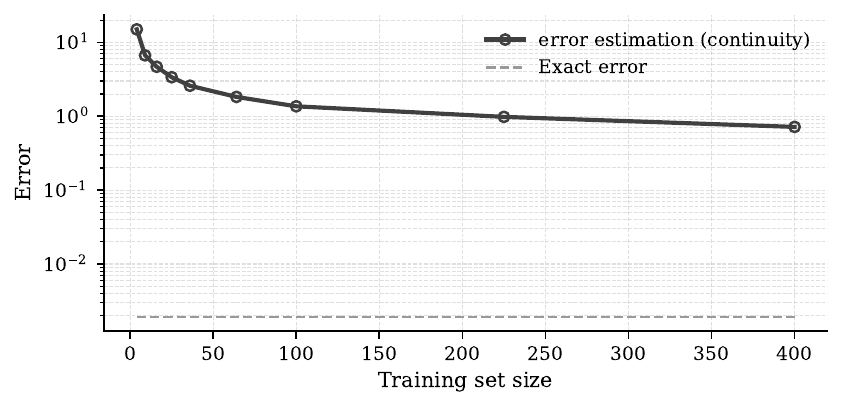}
  \caption{A posteriori error estimation using continuity in function of the size of the training set.}
  \label{fig: estimation continuity}
\end{figure}
We may improve numerically the estimation by computing constant corrections as we did in Section \ref{paragraph:  error correction}. We show the resulting estimation on Figure \ref{fig: continuity error correction}.
\begin{figure}[htbp] 
  \centering
  \includegraphics[width=0.99\linewidth]{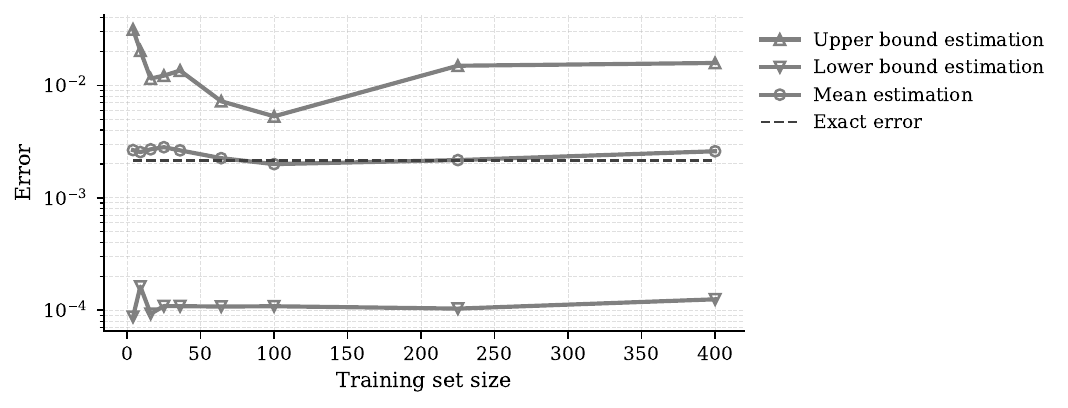}
  \caption{Error correction with a multiplicative constant. The training set used to approximate the constant is of size 100.}
  \label{fig: continuity error correction}
\end{figure}

\subsection{ROM in high dimensions}
The reduction method that we proposed in Section \ref{section: RP} formally requires to build a matrix $\widehat\WW$ of size $\cN_x\cN_y\times R$, which can be numerically challenging when $\cN_x$ and $\cN_y$ are large. We show in this section that  we do not need to compute the matrix $\widehat \WW$ to build the reduced-order model \eqref{eq: RP}. We then propose an efficient way to reconstruct a map from a reduced-order solution. 
\paragraph{Building the reduced-order model without computing \texorpdfstring{$\WW$}{}}
We show here that there is no need to compute explicitly $\widehat{\WW}$ to build the reduced order model \eqref{eq: RP}.
The matrix $\widehat{\WW}$ has the form
\begin{equation}
    \widehat \WW = \begin{pmatrix} \ppi_*(\aalpha_1) | \cdots | \ppi_*(\aalpha_R)\end{pmatrix},
\end{equation}
where $\{ \ppi_*(\aalpha_r) \}_{1 \leq r\leq R}$ are solutions of the high-fidelity problems $\{P_{\aalpha_r}\}_{1 \leq r \leq R}$. In the reduced-order model \eqref{eq: RP}, $\widehat \WW$ appears in the cost in the form $\widehat \WW^T \cc$ and in the constraints in the form $\AA \widehat \WW$. But both of these matrix can be reformulated. On one hand, we have
\begin{equation} \label{eq: reduced cost}
    \widehat \WW^T \cc = \begin{pmatrix}\ppi_*(\aalpha_1)^T \cc | \cdots | \ppi_*(\aalpha_R)^T \cc\end{pmatrix}^T
     = 
     \begin{pmatrix} I(\aalpha_1) | \cdots | I(\aalpha_R)\end{pmatrix}^T, 
\end{equation}
Thus $\widehat \WW^T \cc$ is obtained by stacking the optimal costs of the high-fidelity problems at $\{\aalpha_r\}_{1 \leq r \leq R}$.
On the other hand, 
\begin{equation}\label{eq: reduced constraints}
    \AA \widehat \WW = \begin{pmatrix} \mmu(\aalpha_1) | \cdots | \mmu(\aalpha_R) \\
    \nnu(\aalpha_1) | \cdots | \nnu(\aalpha_R)
    \end{pmatrix},
\end{equation}
thus $\AA \widehat \WW$ can even be obtained without solving any proble.

\paragraph{Reconstructing solutions in high-dimensions from solutions of \texorpdfstring{\eqref{eq: RP}}{}}\label{subsubsection: build map T}
In some applications (see for instance Section \ref{subsection: color transfer}), one may need to build from a solution of the reduced-order model a transport map approximating the solution of the high-fidelity model \eqref{eq: primal}. As show in Figure \eqref{fig: solutions of ROM}, the solutions of the the reduced-order model does not have the shape of a map. 
A similar issue is also encountered when solving Optimal Transport with Sinkhorn method, regularized transport \cite{ferradans2014regularized}  or with sliced optimal transport \cite{pitie2007automated} \cite{rabin2010regularization}.
In \cite{ferradans2014regularized}, the transport map is built from a transport plan $\displaystyle \left\lbrace \PPi_{ij} \right\rbrace_{\begin{matrix} 1 \leq i \leq \cN_x \\ 1 \leq j \leq \cN_y\end{matrix}}$ as follows
\begin{equation}
   \displaystyle  \TT_i := \left\lfloor \frac{\displaystyle\sum_{1 \leq j \leq \cN_y} \PPi_{ij} j}{\displaystyle\sum_{1 \leq j \leq \cN_y} \PPi_{ij} } \right\rfloor \hspace{1cm} \text{for } 1 \leq i \leq \cN_x.
\end{equation}
We adapt this computation to the reduced order model. Indeed, given a solution $\widehat \pp_*(\aalpha)$ of the reduced-order model \eqref{eq: RP}, one can build $\TT_*$ in the following way
\begin{equation}
    \begin{aligned}
      \displaystyle  \TT_{*i} :&=& \left\lfloor \frac{\displaystyle  \sum_{1 \leq j \leq \cN_y} \sum_{1\leq r \leq R} \pp_*^r(\aalpha) \PPi_{*ij}(\aalpha_r) j}
      { \displaystyle \sum_{1 \leq j \leq \cN_y}  \sum_{1\leq r \leq R} \pp_*^r(\aalpha) \PPi_{*ij}(\aalpha_r) j } \right\rfloor \\     
      & =& \left\lfloor \frac{ \displaystyle \sum_{1\leq r \leq R} \pp_*^r(\aalpha)  \left( \sum_{1 \leq j \leq \cN_y}    \PPi_{*ij}(\aalpha_r) j \right)}
      {\displaystyle \sum_{1\leq r \leq R} \pp_*^r(\aalpha) \left( \sum_{1 \leq j \leq \cN_y}  \PPi_{*ij}(\aalpha_r) \right) } \right\rfloor,  
    \end{aligned}
\end{equation}
where $ \displaystyle \{ \PPi_*(\aalpha_r) \}_{1 \leq r \leq R} $ are the matrix versions of the the transport plans $\{ \ppi_*(\aalpha_r) \}_{1 \leq r \leq R} $
This way, after precomputations, we reduce the inital building cost $\cN_x\cN_y$ to $R\cN_x$.
\subsection{Application : Reduced-bases applied to color transfer}\label{subsection: color transfer}
In this section, we present an application of the reduction method to a color transfer problem. The principle of color transfer is to modify an image so that its  color histogram matches a target histogram (named color palette). To do so, \cite{pitie2007automated} and \cite{rabin2010regularization} propose to match the histogram of the original image to the target palette with a map that minimizes the total transport cost between these two histograms.
In the present case, we propose a variation of this problem: given an image and $K_y$ color histograms, we would like to transfer the color histogram of our image to color palettes built as convex combinations of the $K_y$ histograms. We begin by introducing the transport problem solved to perform the color transfer. We then compare the results obtained with the reduced-order model to what is obtained solving the high-fidelity model.

\paragraph{Principle of color transfer}
Given a source image $\mathrm{Im}$, we denote by $\mmu \in \RR^{\cN^3_x}$ its color histogram. Let  $\{\mathrm{Target_l}\}_{1\leq l \leq  K_y}$ be target palettes and $\{ \nnu_l \}_{1\leq l \leq K_y} \in (\RR^{\cN_y^3})^{K_y}$ be their corresponding color histograms. Given a convex combination $\displaystyle \nnu(\aalpha) := \sum_{1 \leq l \leq K_y} \aalpha_l \nnu_l $ of these color histograms, we would like to find a map $T(\aalpha) : \RR^{\cN^3_x} \longrightarrow \RR^{\cN^3_y}$ that transports $\mmu$ onto $\nnu(\aalpha)$ and that minimize the total transport cost :
\begin{equation}
    \mathop{\min}_{\begin{array}{c}
    T : \RR^{\cN^3_x} \longrightarrow \RR^{\cN^3_y} \\
   \displaystyle \sum_{i : T(i) = j} \mmu_i = \nnu(\aalpha)_j, \text{ for } j = (j_1, j_2, j_3), 1 \leq j_1, j_2, j_3 \leq \cN_y
    \end{array} }
    \displaystyle \sum_{\small\begin{array}{c} i = (i_1, i_2, i_3),\\ 1 \leq i_1, i_2, i_3 \leq \cN_x \end{array}} \CC_{i T(i)}, 
\end{equation}
where $\CC : \RR^{\cN^3_x\times  \cN^3_y} \longrightarrow \RR_+$. In practice, for $i = (i_1, i_2, i_3)$ and $j = (j_1, j_2, j_3)$, $\{\CC_{i,j} = |i_1 - j_1|^2 + |i_2 - j_2|^2 + |i_3 - j_3|^2 \}$.
As underlined in \cite{COTFNT}, this discrete Monge problem is possibly not well defined. We relax this problem with the Kantorovitch formulation
\begin{equation}\label{eq: color transfer}
\displaystyle \mathop{\min}_{\begin{array}{cc}
&\PPi \in \RR_+^{\cN^3_x \times \cN^3_y} \\
&\displaystyle \sum_{j} \PPi_{i,j} = \mmu_i \text{ for } i = (i_1, i_2, i_3), 1 \leq i_1, i_2, i_3 \leq \cN_x,\\
&\displaystyle \sum_{i} \PPi_{i,j} = \nnu_j(\aalpha) \text{ for } j = (j_1, j_2, j_3), 1 \leq j_1, j_2, j_3 \leq \cN_y \\
\end{array}
}
\left\langle \PPi, \CC \right\rangle .
\end{equation}
Problem \eqref{eq: color transfer} reads as a 3D and two-marginal optimal transport problem.
To solve it and build a reduced order model, we use the Sinkhorn method designed in  \cite{Storimage}.
Once an optimal transport plan has been built, it is possible to obtain a map following the method described in Section \ref{subsubsection: build map T}.

\paragraph{Results} In the following numeraical tests, $\cN_x = \cN_y = 64$ so that $\cN_x^3 = \cN_y^3 \approx 2.6 \; 10^5$.
Table \ref{tab:time_comparison} compares the time taken on a commodity laptop to solve one transfer color problem (Optimal transport problem + building a map) with Sinkhorn method and with the reduced-basis approach (Here, the reduced-order model is build with only 3 snapshots). The reduced model achieves a significant speed-up over the high-fidelity Sinkhorn method.
\begin{table}[htbp]
\centering
\begin{tabular}{l c c}
\toprule
Method & Time (s) & Speed-up \\
\midrule
Solving Reduced model & 0.0208 & $\times \textbf{333}$ \\
High-fidelity (Sinkhorn) & 6.93 & $1$ \\
\bottomrule
\end{tabular}
\caption{Computation time comparison. The value corresponds to the mean computation time on a test set of 5 parameters.}\label{tab:time_comparison}
\end{table}

On Figure \ref{fig: color transfer}, we compare the results of the reduced basis method to the results obtained via solving the high-fidelity model. At first sight, the images obtained with both methods look similar. On Figure \ref{fig: color transfer detail}, we exhibit that one can still observe differences. Indeed, on this example, the red tones appear less vivid in the picture obtained with the reduced-order model.
\begin{figure}[htbp] 
  \centering
  \includegraphics[width=1\linewidth]{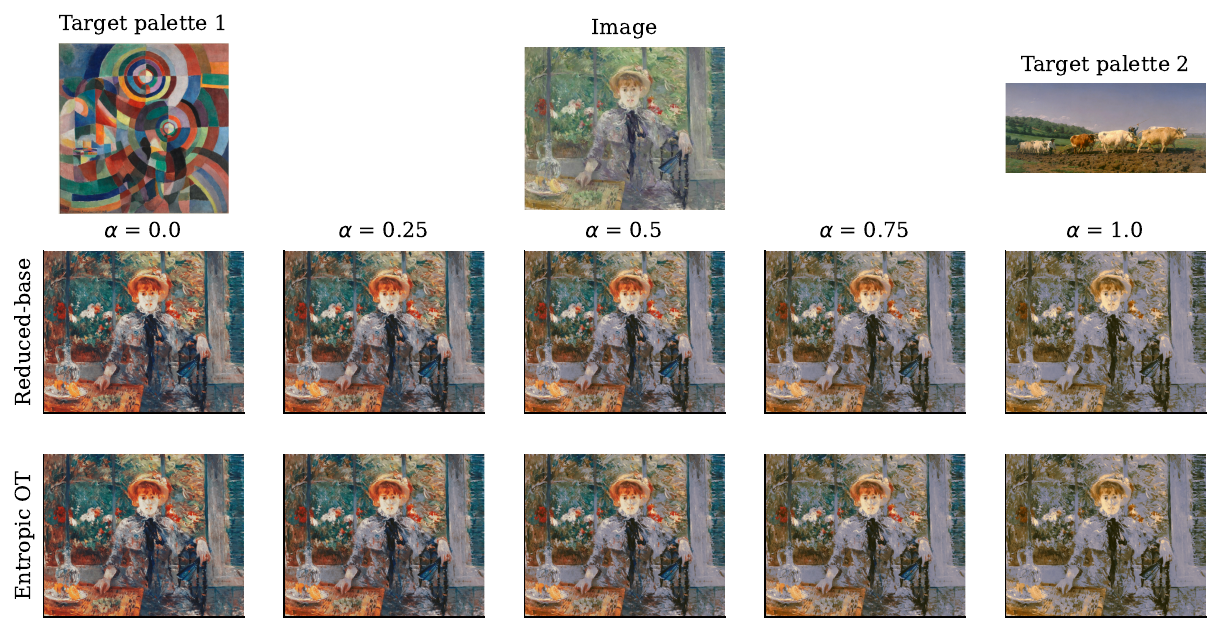}
  \caption{Color transfer obtained with a reduced-basis method (line 2) and via. These pictures where obtained with a reduced basis with only 3 snapshots. For $0 \leq \alpha \leq 1$, the target palette is computed as $(1-\alpha) \texttt{ (target palette 1) } + \alpha \texttt{(target palette 2)}$. \\ \tiny{Sources : Image (Berthe Morisot, \textit{Après le déjeuner}, 1881), Target palette 1 (Sonia Delaunay, \textit{Prismes électriques}, 1914), Target palette 2 (Rosa Bonheur, \textit{Labourage nivernais}, 1849).}{}}
  \label{fig: color transfer}
  
\end{figure}
\begin{figure}[htbp] 
  \centering
  \includegraphics[width=1\linewidth]{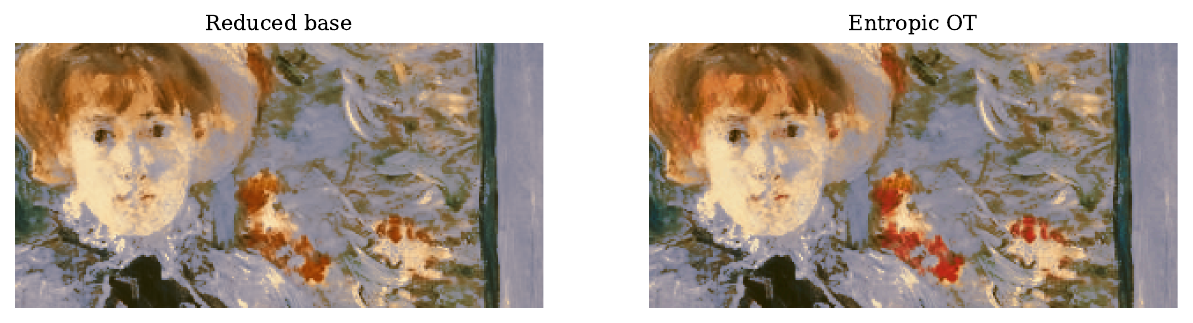}
  \caption{Detail of the resulting images for $\aalpha = 0.9$.}{}
  \label{fig: color transfer detail}
\end{figure}

\paragraph{Acknowledgments}
The third author's work  benefited from the support of the FMJH Program PGMO, Université Paris-Saclay and from the ANR project GOTA (ANR-23-CE46-0001). We acknowledge the financial support of European Research Council (ERC) under the European Union’s
Horizon 2020 Research and Innovation Programme – Grant Agreement n°101077204 HighLEAP.

\bibliographystyle{plain}
\bibliography{biblio}
\end{document}

%% file: macros.tex
\def\RR{\mathbb{R}}
\def\N{\mathbb{N}}

\def\cA{\mathcal{A}}
\def\cG{\mathcal{G}}

\def\cI{\mathcal{I}}

\def\cN{\mathcal{N}}

\def\cU{\mathcal{U}}
\def\cV{\mathcal{V}}
\def\cW{\mathcal{W}}

\def\aalpha{\boldsymbol{\alpha}}
\def\bbeta{\boldsymbol{\beta}}

\def\pphi{\boldsymbol{\varphi}}
\def\ppsi{\boldsymbol{\psi}}
\def\ppi{\boldsymbol{\pi}}
\def\PPi{\boldsymbol{\Pi}}

\def\bbeta{\boldsymbol{\beta}}
\def\mmu{\boldsymbol{\mu}}
\def\nnu{\boldsymbol{\nu}}

\def\aa{\mathbf{a}}
\def\bb{\mathbf{b}}
\def\cc{\mathbf{c}}

\def\ee{\mathbf{e}}
\def\ff{\mathbf{f}}
\def\Gg{\mathbf{g}}
\def\pp{\mathbf{p}}
\def\qq{\mathbf{q}}
\def\uu{\mathbf{u}}
\def\vv{\mathbf{v}}
\def\ww{\mathbf{w}}

\def\AA{\mathbf{A}}
\def\CC{\mathbf{C}}

\def\sS{\mathbf{S}}

\def\TT{\mathbf{T}}
\def\UU{\mathbf{U}}
\def\VV{\mathbf{V}}
\def\WW{\mathbf{W}}

\def\00{\mathbb{0}}
\def\11{\mathbb{1}}

\def\meim{m^{\text{eim}}}

\def\eim{\text{eim}}

%% file: packages.tex
\usepackage[english]{babel}
\usepackage[letterpaper,top=2cm,bottom=2cm,left=3cm,right=3cm,marginparwidth=1.75cm]{geometry}
\usepackage{multicol}
\usepackage{appendix}

\usepackage{color}
\usepackage[dvipsnames]{xcolor}

\usepackage{subfigure}
\usepackage{graphicx}

\usepackage{tikz}

\usepackage{booktabs}

\usepackage{bbold} 
\usepackage{bm}
\usepackage{eucal}

\usepackage{algorithm}
\usepackage{algpseudocode}

\usepackage{multicol}

\usepackage[
colorlinks=true, 
allcolors= Apricot,
frenchlinks = true]
{hyperref}
\hypersetup{frenchlinks=true}

\usepackage{amssymb}
\usepackage{amsmath}
\usepackage{amsthm}
\usepackage{amsfonts}
\usepackage{latexsym}
\usepackage{stmaryrd} 

\newtheorem{thm}{Theorem}[section]
\newtheorem{defi}[thm]{Definition}
\newtheorem{cor}[thm]{Corollary}
\newtheorem{prop}[thm]{Proposition}
\newtheorem{lem}[thm]{Lemma}
\newtheorem{rmk}[thm]{Remark}